\documentclass[11pt]{article}
\usepackage{latexsym,amsmath}
\def\fullpage {
\addtolength{\topmargin}{-2 cm}
\addtolength{\oddsidemargin}{-1.5cm} \addtolength{\textwidth}{+3 cm}
\addtolength{\textheight}{+4 cm}}
\fullpage

\newcommand{\C}[1]{{\protect\cal #1}}

\usepackage{amsthm}
\usepackage{amssymb}
\usepackage{epsfig}
\usepackage{amsmath}
\usepackage{graphicx}

\newcommand{\beq}[1]{\begin{equation}\label{eq:#1}}
\newcommand{\eeq}{\end{equation}}

\begin{document}
\newtheorem{theorem}{Theorem}
\newtheorem{corollary}[theorem]{Corollary}
\newtheorem{lemma}[theorem]{Lemma}
\newtheorem{claim}[theorem]{Claim}
\newtheorem{proposition}[theorem]{Proposition}
\newtheorem{conjecture}[theorem]{Conjecture}
\newcommand\eps{\varepsilon}

\newcommand{\brac}[1]{\left(#1\right)}
\newcommand{\bfrac}[2]{\brac{\frac{#1}{#2}}}
\newcommand{\rdown}[1]{\left\lfloor#1\right\rfloor}
\newcommand{\rdup}[1]{\left\lceil#1\right\rceil}
\newcommand{\me}{\mathrm{e}}
\newcommand{\ee}{\epsilon}
\newcommand{\ex}{\mathrm{ex}}
\newcommand{\Bin}{\mathrm{Bin}}
\newcommand{\cG}{{{\cal G}}}
\newcommand{\cK}{{\cal K}}
\newcommand{\cH}{{\cal H}}
\newcommand{\cP}{{\cal P}}
\newcommand{\cJ}{{\cal J}}
\newcommand{\cE}{{\cal E}}
\newcommand{\cB}{{\cal B}}
\newcommand{\cM}{{\cal M}}
\def\Q{\mathcal{Q}}
\def\C{\mathcal{C}}
\newtheorem{definition}[theorem]{Definition}
\def\FF{\mathcal{F}}
\def\NN{\mathcal{N}}
\def\F{\mathcal{F}}
\def\S{\mathcal{S}}
\def\SS{\mathcal{S}}
\def\eps{\varepsilon}
\newcommand{\eee}{{\mathbb E}}
\def\HH{\mathcal{H}}
\parindent=0pt

\title{Almost all triple systems with independent neighborhoods
are semi-bipartite}

\author{J\'ozsef Balogh\thanks{{Department of
Mathematics, U.C. California at San Diego, 9500 Gilmann Drive, La
Jolla, Department of Mathematics; and  University of Illinois, 1409
W. Green Street, Urbana, IL 61801, USA}; e-mail:
{jobal@math.uiuc.edu;} research  supported in part   by NSF CAREER
Grant DMS-0745185 and DMS-0600303, UIUC Campus Research Board Grants
09072 and 08086, and OTKA Grant K76099.}  \quad  and \quad Dhruv
Mubayi\thanks{Department of Mathematics, Statistics, and Computer
Science, University of Illinois at Chicago, IL 60607;  email:
mubayi@math.uic.edu; research  supported in part by  NSF grant DMS
0653946.}}

\maketitle

\vspace{-0.4in}

\begin{abstract}
The neighborhood of a pair of vertices $u,v$ in a triple system is the set of vertices $w$ such that $uvw$ is an edge.
 A triple system
 $\HH$ is semi-bipartite if its vertex set
contains a vertex subset $X$ such that every edge of $\HH$
intersects $X$ in exactly two points. It is easy to see that if
$\HH$ is semi-bipartite, then the neighborhood of every pair of
vertices in $\HH$  is an independent set. We show a partial converse
of this statement by proving that almost all triple systems with
 vertex sets $[n]$ and independent neighborhoods are semi-bipartite. Our result can be
viewed as an extension
 of the Erd\H os-Kleitman-Rothschild theorem to triple systems.

The proof uses the Frankl-R\"odl hypergraph regularity lemma, and stability
theorems. Similar results have recently been proved for hypergraphs with various other local constraints.
\end{abstract}

\section{Introduction}

This is the second in a sequence of papers where we describe the global structure of typical $k$-uniform hypergraphs
($k$-graphs for short) that satisfy certain local conditions. This line of research originated with the seminal result of
 Erd\H os-Kleitman-Rothschild \cite{EKR} which proved that almost all triangle-free graphs with vertex set $[n]$ are bipartite.
  Our goal is to prove a hypergraph version of this theorem.

Subsequent to  \cite{EKR},
there has been much work concerning the number and structure of
$F$-free graphs with vertex set $[n]$ (see, e.g. \cite{EFR, KPR,
PS1, BBS1, BBS2, BBS3}). The results essentially state that for a
large class of graphs $F$, most of the $F$-free graphs with vertex
set $[n]$ have a similar structure to the $F$-free graph with the
maximum number of edges. Many of these results use  the Szemer\'edi
Regularity Lemma.

With the recent development of the hypergraph regularity Lemma, one can prove similar theorems for hypergraphs.
 For brevity, we refer to a $3$-uniform hypergraph as a triple system or 3-graph.
 The first result in this direction was due to Nagle and R\"odl \cite{NR} who proved that the
  number of $F$-free triple systems (for fixed triple system $F$) on vertex set $[n]$ is
$$2^{ex(n, F) + o(n^3)},$$
where $ex(n,F)$ is the maximum number of edges in an $F$-free triple
system on $n$ vertices.
 Due to the absence of a general extremal result for hypergraphs in the vein of Tur\'an's graph theorem,
  one cannot expect hypergraph results that completely parallel the graph case. Still, there has been recent progress on various specific examples. Person and Schacht~\cite{PSch} proved that almost all
   triple systems on $[n]$ not containing a Fano configuration are $2$-colorable.
   The key property that they used was the linearity of the Fano plane,
   namely the fact that every two edges of the Fano plane share at most one vertex.
   This enabled them to apply the (weak) $3$-graph regularity lemma, which is almost
   identical to Szemer\'edi's Regularity Lemma.  They then proved an embedding lemma for linear
    hypergraphs essentially following ideas from Kohayakawa-Nagle-R\"odl-Schacht \cite{KNRS}.

It is well-known that such an embedding lemma fails to hold for
non-linear $3$-graphs unless one uses the (strong) $3$-graph
regularity lemma, and operating in this environment is more
complicated.

The first fine result for non-linear hypergraphs was due to the
current authors \cite{t5}. It was proved in \cite{t5} that typical  extended triangle-free
triple systems are tripartite, where an extended triangle is
$\{abc,abd,cde\}$.  The corresponding extremal result, that the maximum number of triples on $[n]$ with no extended
triangle is achieved by a complete tripartite triple system, was proved by Bollob\'as \cite{bollobas:74} and is  the first extremal hypergraph result for a nondegenerate problem.
  In this paper we give a similar result for a different non-linear triple system.

The neighborhood of a $(k-1)$-set $S$ of vertices in a $k$-graph is the set of vertices
$v$ whose union with $S$ forms an edge. A set is independent if it contains no edge.
We can rephrase Mantel's theorem about triangle-free graphs as follows: the maximum number of
 edges in an $n$ vertex  2-graph with independent neighborhoods is $\lfloor n^2/4 \rfloor$.
 This formulation can be generalized to $k>2$ and there has been quite a lot of recent
 activity on this question (\cite{MR, FPS, FMP, BFMP}).

 Let us first observe that a triple system has independent neighborhoods if and only if it contains no copy of
$$T_5=\{123, 124, 125, 345\}.$$  Say that a triple system is {\em semi-bipartite} if it has an (ordered)
 vertex partition $(X,Y)$ such that every edge has exactly one point in $Y$. Then a short case analysis shows that
  all neighborhoods in a semi-bipartite triple system are independent (one can think of semi-bipartite triple systems
   as an analogue of bipartite graphs). Let $B^3(n)$ be the $3$-graph with
the maximum number of edges among all $n$ vertex semi-bipartite
triple systems. Note that
$$b^3(n):=|B^3(n)|=\max_a {a \choose 2}(n-a)=(4/9+o(1)){n \choose 3}$$
is achieved by choosing $a=\lfloor 2n/3 \rfloor$ or $a=\lceil 2n/3
\rceil$.

The second author and R\"odl \cite{MR} conjectured, and F\"uredi,
Pikhurko, and Simonovits \cite{FPS} proved, that among all $n$
vertex $3$-graphs ($n$ sufficiently large) containing no
 copy of  $T_5$, the unique one with the maximum number of edges is $B^3(n)$.

Let  $\SS(n)$ be the set of (labeled)  semi-bipartite $3$-graphs
with vertex set  $[n]$ and put $S(n):=|\SS(n)|$. Let $I(n)$ be the
number of (labeled) $3$-graphs with vertex set $[n]$ and independent
neighborhoods, by which we mean that for every $x,y\in [n]$ there is
no $e \in \HH$ with $e \subset \{z:\ xyz\in \HH\}$.
 Our main result, which is a possible extension of the Erd\H os-Kleitman-Rothschild
 theorem to triple systems, is the following:

\begin{theorem}\label{maint}
Almost all triple systems with independent neighborhoods and vertex
set $[n]$ are semi-bipartite. More precisely there is a constant $C$
such that \begin{equation}\label{approxmain}  (1+
2^{-4n})S(n)\ <\ I(n)\ < \ (1+C\cdot
2^{-n/10})S(n).\end{equation}\end{theorem}

\section{Broad proof structure}

The lower bound in Theorem \ref{maint} will be proved by constructing a large class of triple systems that are not
 semi-bipartite but yet have independent neighborhoods. This will be done  in  Section~\ref{lower}.
  The majority of the paper is devoted to proving the upper bound in Theorem~\ref{maint}. We will do this in two stages.
   First, we will prove that a large majority of  triple systems with vertex set $[n]$
and independent neighborhoods are very close to being semi-bipartite.  This is formalized in Theorem~\ref{stablet} below.
 After this, we can confine our attention to triple systems with independent neighborhoods that are close to being semi-bipartite.
  We will show (see Theorem~\ref{clean}) that most of these triple systems are semi-bipartite.
Let us proceed more formally.

For a hypergraph $F$ let  $Forb(n, F)$ denote the set of   $F$-free
hypergraphs on vertex set $[n]$. Let $P=(X, Y)$ be an ordered vertex
partition of a 3-graph $\HH$. Call an edge of $\HH$ {\em consistent}
with $P$ if it has exactly two points in $X$, otherwise call it {\em
inconsistent}. Let $D_P$ be the set of inconsistent edges with $P$.
A vertex partition $P$ is {\em optimal} for $\HH$ if it minimizes
the number of inconsistent edges, and let $D=D_{\HH}$  be the number
of inconsistent edges in an optimal partition of $\HH$. Define

$$Forb(n, T_5, \eta):=\left\{\HH \subset \binom{[n]}{3}: T_5 \not\subset \HH\hbox{ and } D_{\HH}\le \eta n^3\right\}.$$

The  proof of Theorem \ref{maint} can be separated into two parts;
 Theorem~\ref{stablet}, proved in Section 4 and 5 and Theorem~\ref{clean},
 proved in Section 6. Note that the proof of Theorem~\ref{stablet} is independent from the rest of the
 results. However, both Theorems~\ref{maint} and \ref{clean} are proved via induction on $n$: In the proof of the $n$-statement of Theorem~\ref{maint} we use the $n'$-statement of Theorem~\ref{clean} for every $n'\le n$, and in the proof of the $n$-statement of Theorem~\ref{clean} we use the $n'$-statement of Theorem~\ref{stablet} for every $n'< n$.  This will be made more precise in Section 6.6.

\begin{theorem} \label{stablet}
For every $\eta>0$, there exists $\nu>0$ and $n_0$ such that
 if $n>n_0$, then $$|Forb(n,T_5) - Forb(n, T_5, \eta)|<2^{(1-\nu)\frac{2n^3}{27}}.$$
\end{theorem}

 We will use the hypergraph regularity lemma due
to Frankl-R\"odl to prove Theorem~\ref{stablet}. In
Section~\ref{hypreg} we introduce the definitions needed to state
this lemma.

\begin{theorem} \label{clean}
For $\eta>0$ sufficiently small there exists a $C'$ such that
\begin{equation}\label{cleanin}|Forb(n, T_5, \eta)|\ <\ (1+C'2^{-n/10}) S(n).\end{equation}
\end{theorem}

The proof of Theorem~\ref{clean} uses many ideas from
\cite{BBS1,BBS2}: we prove in Section~\ref{sublower}  that most $\HH\in Forb(n,
T_5, \eta)$ have some lower-dense properties, in Section~\ref{subbadvert} that
there are no vertices with many inconsistent edges, and in \ref{secinconedges} we
shall get rid of  all the inconsistent edges. However many elements
of the proof are new, like the using the concept of rich
 edges and the shadow graphs.

\section{Lower bound in Theorem~\ref{maint}}\label{lower}

Let us prove the lower bound in \eqref{approxmain},  by constructing a set $\NN\SS(n)$ of at least $2^{-4n}S(n)$ non-semi-bipartite
$T_5$-free 3-graphs $\HH$ with vertex set $[n]$.  Indeed, this shows that $I(n)-S(n) \ge 2^{-4n}S(n)$ and it follows that $I(n)>(1+2^{-4n})S(n)$.

  Let $s=s(n)$ be the maximum number of
edges that a semi-bipartite 3-graph with vertex set $[n]$ can have, and suppose that this is achieved with class sizes
$t=t(n)$ and $n-t$ (where $t \ge n-t$). Easy calculus shows that $t<2n/3 +2$. Then clearly $$S(n)\le 2^{n+s}.$$
 Let $X=[t]$ and $Y= [n]-[t]$. Set
$${\cal F}= {\{1,2,n-1, n\}\choose 3}.$$
Let $\cal G$ be the collection of triples $e$ that simultaneously satisfy the following two conditions:

$\bullet$ $|e \cap X|=2$,

$\bullet$  $|e \cap \{1,2,n-1,n\}| \le 1$. \quad (*)

Let $\NN\SS(n)$ be
the collection of 3-graphs
$\{{\cal F} \cup {\cal G}': {\cal G}' \subset {\cal G}\}.$
 We will now show that $\NN\SS(n)$ comprises only non-semi-bipartite $T_5$-free 3-graphs.
 Pick an $\HH \in \NN\SS(n)$.

Since ${\cal F}$ is not semi-bipartite, $\HH$ is also not
semi-bipartite. Using (*), an easy case analysis shows that $T_5
\not\subset \HH$. Finally, we must obtain a lower bound on
$|\NN\SS(n)|=2^{|{\cal G}|}$.  Recall that $s={t \choose 2}(n-t)$.
Since we exclude all triples with two or more points in
$\{1,2,n-1,n\}$ when defining $\cal G$, and $t \le 2n/3+2$,
$$|{\cal G}| =s-(n-t+4(t-2)+2)\ge -3n+s=-4n+n+s.$$
Consequently,
$$|\NN\SS(n)|=2^{|{\cal G}|}\ge 2^{-4n} 2^{n+s}\ge 2^{-4n} S(n)$$
and the proof is complete.

\section{Hypergraph Regularity}\label{hypreg}

We quickly define the notions required to state the hypergraph
regularity Lemma.  Details can be found in \cite{NR}.  Throughout we
associate a hypergraph with its edge set.

A $k$-{\it partite cylinder} is a $k$-partite graph $G$ with
$k$-partition $V_1, \ldots, V_k$, and we write $G=\cup_{i<j}
G^{ij}$, where $G^{ij}=G[V_i \cup V_j]$ is the bipartite subgraph
 of $G$ with parts $V_i$ and $V_j$.
 For $B \in \binom{[k]}{3}$, the $3$-partite cylinder $G(B)=\cup_{\{i,j\} \in [B]^2} G^{ij}$ is called a {\it triad}.
  For a $2$-partite cylinder $G$,
 the {\it density} of the pair $(V_1, V_2)$ with respect to $G$ is $d_G(V_1, V_2)=\frac{|G|}{|V_1||V_2|}$.

Given  an integer $l>0$ and real $\epsilon>0$, a $k$-partite
cylinder $G$ is called an $(l, \epsilon, k)$-{\it cylinder} if for
every $i<j$, $G^{ij}$ is $\epsilon$-regular with density $1/l$.
 For a $k$-partite cylinder $G$, let $\cK_3(G)$ denote the
 $3$-graph on $V(G)$ whose edges correspond to triangles of $G$. An easy
consequence of these definitions is the following fact.

\begin{lemma} {\bf (Triangle Counting Lemma)}  \label{tlemma} For any integer $l>0$ and  real $\theta>0$,
 there exists an $\epsilon>0$ such that every $(l,\epsilon,3)$-cylinder $G$ with $|V_i|=m$ for all $i$ satisfies
$$|\cK_3(G)| =(1\pm \theta)\frac{m^3}{l^3}.$$
\end{lemma}

We now move on to $3$-graph definitions. A $k$-{\it partite
$3$-cylinder} is a $k$-partite $3$-graph $\cH$ with $k$-partition
$V_1, \ldots, V_k$. Here $k$-partite means that every edge of $\cH$
has at most one point in each $V_i$.  Often we will say that these
edges are {\it crossing}, and the edges that have at least two
points in some $V_i$ are {\it non-crossing}. Given $B \in
\binom{[k]}{3}$, let $\cH(B)=\cH[\cup_{i \in B}V_i]$. Given a
$k$-partite cylinder $G$ and $k$-partite $3$-cylinder $\cH$ with the
same vertex partition, say that $G$ {\it underlies} $\cH$ if $\cH
\subset \cK_3(G)$.  In other words, $\cH$ consists  only of
triangles in $G$. When $G$ underlies $\cH$, define the density
$d_{\cH}(G(B))$ of $\cH$ with respect to the triad $G(B)$ as the
proportion of edges of $\cH$ on top of triangles of $G(B)$, if the
latter quantity is positive, and zero otherwise. This definition
leads to the more complicated definition of $\cH$ being $(\delta,
r)$-regular with respect to $G$ or $G(B)$, where $r>0$ is an integer
and $\delta>0$. If in addition $d_{\cH}(G)=\alpha \pm \delta$, then
say that $\cH$ is $(\alpha, \delta, r)$-{\it regular} with respect
to $G$. We will not give the precise definition of $(\alpha, \delta,
r)$-regularity, and it suffices to take this definition as a ``black
box" that will be used later.

For a vertex set $V$, an $(l, t, \gamma, \epsilon)$-partition
 $\cP$ of $\binom{[V]}{2}$ is a partition $V=V_0 \cup V_1 \cup \cdots \cup V_t$
together with a collection of edge disjoint bipartite graphs $P_{a}^{ij}$,
 where $1\le i<j\le t,\  0\le a\le l_{ij} \le l$ that satisfy the following properties:

(i) $|V_0|<t$ and $|V_i|=\lfloor \frac{n}{t} \rfloor:=m$ for each
$i>0$,

(ii) $\cup_{\alpha=0}^{l_{ij}}P_{\alpha}^{ij}=K(V_i, V_j)$ for all
$1\le i<j\le t$, where $K(V_i, V_j)$ is the complete bipartite graph
with parts $V_i, V_j$,

(iii) all but $\gamma{t \choose 2}$ pairs $\{v_i, v_j\}$, $v_i \in V_i, v_j \in V_j$,
are edges of $\epsilon$-regular bipartite graphs $P_{\alpha}^{ij}$, and

(iv) for all but $\gamma{t \choose 2}$ pairs $\{i,j\} \in
\binom{[t]}{2}$, we have $|P_0^{ij}|\le \gamma m^2$ and $d_{
P_{\alpha}^{ij} }(V_i, V_j)=(1\pm \epsilon)\frac{1}{l}$ for all
$\alpha \in [l_{ij}]$.

Finally, suppose that $\cH \subset \binom{[n]}{3}$ is a $3$-graph
and $\cP$ is an $(l, t, \gamma, \epsilon)$-partition of
$\binom{[n]}{2}$ with $m_{\cP}=|V_1|$. For each triad $P \in \cP$,
let $\mu_P=\frac{|\cK_3(P)|}{m_{\cP}^3}$. Then $\cP$ is $(\delta,
r)$-regular if
$$\sum\{\mu_P: \hbox{$P$ is a $(\delta, r)$-irregular triad of $\cP$}\} <\delta\left(\frac{n}{m_{\cP}}\right)^3.$$

We can now state the Regularity Lemma due to Frankl and R\"odl
\cite{FR}.

\begin{theorem} {\bf (Regularity Lemma)} \label{rl}
For every $\delta, \gamma$ with $0<\gamma\le 2\delta^4$, for all
integers $t_0, l_0$ and for all integer-valued functions $r=r(t, l)$
and all functions $\epsilon(l)$, there exist $T_0, L_0, N_0$ such
that every 3-graph $\cH \subset \binom{[n]}{3}$ with $n\ge N_0$
admits a $(\delta, r(t,l))$-regular $(l, t, \gamma,
\epsilon(l))$-partition for some $t,l$ satisfying $t_0 \le t<T_0$
and $l_0\le l<L_0$.
\end{theorem}

To apply the Regularity Lemma above, we need to define a cluster hypergraph and state an accompanying embedding Lemma,
 sometimes called the Key Lemma.  Given a $3$-graph $\cJ$, let $\cJ^2$ be the set of pairs that lie in an edge of $\cJ$.

{\bf Cluster $3$-graph.} For given constants $k, \delta, l, r,
\epsilon$ and sets $\{\alpha_B: B \in \binom{[k]}{3}\}$ of
non-negative reals,
 let $\cH$ be a $k$-partite $3$-cylinder with parts $V_1, \ldots, V_k$, each of size $m$.  Let $G$ be a graph, and
 $\cJ \subset \binom{[k]}{3}$ be a $3$-graph such that the following conditions are satisfied.

(i) $G=\cup_{\{i,j\} \in \cJ^2} G^{ij}$ is an underlying cylinder of
$\cH$ such that for all $\{i,j\} \in \cJ^2$, $G^{ij}$ is an $(l,
\epsilon, 2)$-cylinder.

(ii) For each $B \in \cJ$, $\cH(B)$ is $(\alpha_B, \delta, r)$-regular with respect to the triad $G(B)$.

Then we say that $\cJ$ is the {\it cluster $3$-graph} of $\cH$.

\begin{lemma} {\bf (Embedding Lemma)} \label{elemma} Let $k \ge 4$ be fixed.
 For all  $\alpha>0$, there exists $\delta>0$ such that for $l>\frac{1}{\delta}$, there exists
  $r, \epsilon$ such that the following holds: Suppose that $\cJ$ is the cluster $3$-graph
  of $\cH$ with underlying cylinder $G$ and parameters $k, \delta, l, r, \epsilon, \{\alpha_B: B \in \binom{[k]}{3}\}$
   where $\alpha_B \ge \alpha$ for all $B \in \cJ$.  Then $\cJ \subset \cH$.
\end{lemma}

For a proof of the Embedding Lemma, see \cite{NR}.

\section{Proof of Theorem \ref{stablet}}\label{sectionproofstab}

In this section we prove Theorem \ref{stablet}.
We will need the following stability result proved in \cite{FPS}.
The constants have been adjusted for later use.

\begin{theorem} {\bf (F\"uredi-Pikhurko-Simonovits \cite{FPS})} \label{fps}
For every $\nu''>0$, there exist $\nu_1', t_2$ such that every
$T_5$-free $3$-graph on $t>t_2$ vertices and at least
$(1-2\nu_1')\frac{2t^3}{27}$ edges has an ordered partition for
which the number of inconsistent edges is at most $\nu'' t^3$.
Additionally, there exists $t_3$ such that ex$(n, T_5)\le
\frac{2t^3}{27}$ for all $t \ge t_3$.
\end{theorem}

Given $\eta>0$, our constants will obey the following hierarchy:
$$\eta\gg \nu''\gg \nu' \gg \nu \gg \sigma, \theta
\gg  \alpha_0, \frac{1}{t_0} \gg \delta \gg \gamma >\frac{1}{l_0} \gg\frac{1}{r}, \epsilon \gg \frac{1}{n_0}.$$
Before proceeding with further details regarding our constants,
we define the {\it binary entropy function} $H(x):=
-x\log_2 x- (1-x)\log_2 (1-x).$  We use the following two  facts about $H(x)$ that apply for $n$ sufficiently large:

$\bullet$ for $0<x< 0.5$ we
have $$\binom{n}{\lfloor xn\rfloor}<2^{H(x)n}.$$

$\bullet$  if $x$ is sufficiently small
then
\begin{equation} \label{x} \sum_{i=0}^{\lfloor xn\rfloor} \binom{n}{i}<2^{H(x)n}.\end{equation}

{\bf Detailed definition of constants.}

Set
\begin{equation} \label{nu''def}\nu''=\left(\frac{\eta}{30}\right)^3\end{equation}
 and suppose that $\nu'_1$ and $t_2$ are the outputs of Theorem \ref{fps} with input $\nu''$.  Put
\begin{equation} \label{nu'}
\nu'=\min\left\{\nu'_1, \frac{\nu''}{2}, \frac{\eta}{7}\right\} \quad \hbox{ and } \quad \nu=(\nu')^4.\end{equation}
We choose
  \begin{equation} \label{theta}
  \theta=\frac{\nu}{4(1-\nu)}.\end{equation}

Choose $\sigma_1$ small enough  so that
\begin{equation} \label{sigma}
\left(1-\frac{\nu}{2}\right)\frac{2n^3}{27}+o(n^3)+H(\sigma_1)n^3\le
\left(1-\frac{\nu}{3}\right)\frac{2n^3}{27}\end{equation} holds for
sufficiently large $n$.  In fact the function denoted by $o(n^3)$
will actually be seen to be of order $O(n^2)$ so (\ref{sigma}) will
hold for sufficiently large $n$. Choose $\sigma_2$ small enough so
that (\ref{x}) holds for $x=\sigma_2$. Let
$$\sigma=\min\left\{\sigma_1, \sigma_2, \frac{\eta}{2}\right\}.$$
Next we consider the Triangle Counting Lemma (Lemma \ref{tlemma}) which provides an
$\epsilon$ for each  $\theta$ and $l$. Since $\theta$ is fixed, we may let
$\epsilon_1=\epsilon_1(l)$ be the output of Lemma \ref{tlemma} for each integer $l$.

For $\sigma$ defined above, set
\begin{equation} \label{alpha}\delta_1=\alpha_0=\frac{\sigma}{100} \quad
\hbox{ and } \quad t_1=\left\lceil \frac{1}{\delta_1} \right\rceil.\end{equation}
Let
$$t_0=\max\{t_1, t_2, t_3\}.$$
Now consider the Embedding Lemma (Lemma \ref{elemma}) with inputs $k=5$ and $\alpha_0$
defined above.  The Embedding Lemma gives $\delta_2=\delta_2(\alpha_0)$, and
we set
\begin{equation} \label{delta} \delta=\min\{\delta_1, \delta_2\}, \quad
\quad \gamma=\delta^4, \quad \quad l_0=\frac{2}{\delta}.\end{equation}
For each integer $l>\frac{1}{\delta}$, let $r=r(l)$ and
$\epsilon_2=\epsilon_2(l)$ be the outputs of Lemma~\ref{elemma}. Set
\begin{equation} \label{epsilonl}\epsilon=\epsilon(l)=\min\{\epsilon_1(l), \epsilon_2(l)\}.\end{equation}

With these constants, the Regularity Lemma (Theorem \ref{rl}) outputs $N_0$.  We choose
$n_0$ such that $n_0>N_0$  and every $n>n_0$ satisfies
(\ref{x}) and (\ref{sigma}).

\medskip

{\bf Proof of the Theorem \ref{stablet}.}

We will prove that $$|Forb(n, T_5) - Forb(n, T_5, \eta)|<2^{(1-\frac{
\nu}{3})\frac{2n^3}{27}}.$$  This is of course equivalent to Theorem
\ref{stablet}.

For each $\HH \in Forb(n, T_5) - Forb(n, T_5, \eta)$, we use the Hypergraph
Regularity Lemma, Theorem~\ref{rl}, to obtain a $(\delta,
r)$-regular $(l, t, \gamma, \epsilon)$-partition $\cP=\cP_{\HH}$.
The input constants for Theorem~\ref{rl} are as defined above, and then Theorem \ref{rl} guarantees constants $T_0,
L_0, N_0$ so that every $3$-graph $\HH$ on $n>N_0$ vertices admits a
$(\delta, r)$-regular $(l, t, \gamma, \epsilon)$-partition $\cP$
where $t_0 \le t \le T_0$ and $l_0\le l \le L_0$. To this partition $\cP$,
associate a {\em density vector} $s=(s_{\{i,j,k\}_{a,b,c}})$ where $1
\le i<j<k\le t$ and $1\le a,b,c \le l$ and
$$d_{\cH}(P_a^{ij} \cup P_b^{jk} \cup P_c^{ik})\in [s_{\{i,j,k\}_{a,b,c}}\delta, (s_{\{i,j,k\}_{a,b,c}}+1)\delta].$$
For each $\HH \in
 Forb(n, T_5) - Forb(n, T_5, \eta)$, choose one $(\delta, r)$-regular $(l, t,
\gamma, \epsilon)$-partition $\cP_{\HH}$ guaranteed by Theorem~\ref{rl}, and let $\cP=\{\cP_1, \ldots, \cP_p\}$ be the set of all
such partitions over the family $Forb(n, T_5) -Forb(n, T_5, \eta)$. Define an
equivalence relation on $Forb(n, T_5) -Forb(n, T_5, \eta)$ by letting $\HH\sim
\HH'$ iff

1) $\cP_{\HH}=\cP_{\HH'}$ and

2) $\HH$ and $\HH'$ have the same density vector.

The number of equivalence classes $q$ is the number of partitions
times the number of  density vectors.  Consequently,
$$q\le \left({T_0 +1\choose 2}(L_0+1)\right)^{n \choose 2}
\left(\frac{1}{\delta}\right)^{{T_0+1 \choose 3}(L_0+1)^3}<2^{O(n^2)}.$$

 We will show that each equivalence class $C(\cP_i)$ satisfies
\begin{equation} \label{C} |C(\cP_i)|=2^{(1-\frac{\nu}{2})\frac{2n^3}{27}+H(\sigma) n^3}.\end{equation}
Combined with the upper bound for $q$ and (\ref{sigma}), we
obtain
$$|Forb(n, T_5) -Forb(n, T_5, \eta)|\le 2^{O(n^2)}2^{(1-\frac{\nu}{2})\frac{2n^3}{27}+H(\sigma)n^3}\le
2^{(1-\frac{\nu}{3})\frac{2n^3}{27}}.$$

For the rest of the proof, we fix an equivalence class $C=C(\cP)$
and we will show the upper bound in (\ref{C}).  We may assume that
$\cP$ has vertex partition $[n]=V_0\cup V_1\cup \cdots \cup V_t$,
$|V_i|=m=\lfloor \frac{n}{t}\rfloor$ for all $i\ge 1$, and system of
bipartite graphs $P_{a}^{ij}$, where $1\le i<j\le t, 0\le a\le
l_{ij} \le l$.

 Fix $\HH \in C$.  Let $\cE_0\subset \HH$ be the set of triples that either

 (i) intersect $V_0$, or

 (ii) have at least two points in some
$V_i, i\ge 1$, or

 (iii) contain a pair in $P_0^{ij}$ for some $i<j$, or

 (iv) contain a pair in some $P_{a}^{ij}$ that is not $\epsilon$-regular with density $\frac1l$.

 Then
 $$|\cE_0|\le tn^2+t\left(\frac{n}{t}\right)^2 n +\gamma{t \choose 2}n+2\gamma{t \choose 2}\left(\frac{n}{t}\right)^2 n.$$

  Let $\cE_1 \subset \HH-\cE_0$ be the set of triples $\{v_i, v_j, v_k\}$ such that either

  (i) the three bipartite graphs of $\cP$  associated with the pairs within the triple form
  a triad $P$ that is not $(\delta, r)$-regular with respect to $\HH(\{i,j,k\})$, or

  (ii) the density $d_{\HH}(P)<\alpha_0$.

 Then
 $$|\cE_1|\le 2\delta t^3\left(\frac{n}{t}\right)^3(1+\theta) +\alpha_0
 {t \choose 3}l^3\left(\frac{n}{t}\right)^3 \frac{1}{l^3}.$$

  Let $\cE_{\HH}=\cE_0\cup \cE_1$.
  Now (\ref{alpha}) and (\ref{delta}) imply that
\begin{equation}\label{rhoupp}|\cE_{\HH}|\le \sigma n^3.\end{equation} Set
$\HH'=\HH-\cE_{\HH}$.

Next we define $\cJ^C=\cJ^C(\HH)\subset \binom{[t]}{3} \times [l]
\times [l] \times [l]$ as follows: For $1 \le i <j <k \le t, 1 \le
a,b,c \le l$,  we have $\{i,j,k\}_{a,b,c} \in \cJ^C$ if and only if

(i) $P=P_a^{ij} \cup P_b^{jk} \cup P_c^{ik}$ is an $(l, \epsilon,
3)$-cylinder, and

(ii) $\HH'(\{i,j,k\})$ is $(\overline{\alpha}, \delta, r)$-regular
with respect to $P$, where $\overline{\alpha}\ge \alpha_0$.

We view $\cJ^C$ as a multiset of triples on $[t]$. For each
$\phi:\binom{[t]}{2} \rightarrow [l]$, let $\cJ_{\phi}\subset \cJ^C$ be the
$3$-graph on $[t]$ corresponding to the function $\phi$ (without
parallel edges). In other words, $\{i,j,k\} \in \cJ_{\phi}$ iff the
triples of $\HH$ that lie on top of the triangles of $P_{a}^{ij}
\cup P_{b}^{jk} \cup P_{c}^{ik}$, $a=\phi(ij), b=\phi(jk),
c=\phi(ik)$, are $(\overline{\alpha}, \delta, r)$-regular and the
underlying bipartite graphs $P_{a}^{ij}, P_{a}^{jk}, P_{c}^{ik}$ are
all $\epsilon$-regular with density $1/l$.

By our choice of constants in (\ref{delta}) and (\ref{epsilonl}),
 we see that $\cJ_{\phi}$ is a cluster $3$-graph for $\HH$, and
  hence by the Embedding Lemma $\cJ_{\phi} \subset \HH$. Since $T_5 \not\subset \HH$, we conclude that $T_5 \not\subset \cJ_{\phi}$.
   As it was shown in \cite{FPS} that for $t \ge t_3$, we have ex$(t, T_5)\le \frac{2t^3}{27}$,
     we conclude that
     \begin{equation}\label{jphi}
    |\cJ_{\phi}|\le \frac{2t^3}{27}\end{equation}
for each $\phi :{[t]\choose 2} \rightarrow [l]$.  Recall from (\ref{nu'}) that $\nu'=\nu^{1/4}$.

\begin{lemma}\label{fact} Suppose that $|\cJ^C|>(1-\nu){2l^3t^3}/{27}$.  Then for at least $(1-\nu') l^{{t \choose 2}}$
functions $\phi:\binom{[t]}{2} \rightarrow [l]$ we have
$$|\cJ_{\phi}| \ge (1-\nu')\frac{|\cJ^C|}{l^3}.$$
\end{lemma}

\begin{proof}  Form the following bipartite graph: the vertex partition is
$\Phi \cup \cJ^C$ , where
$$\Phi=\left\{\phi: {[t]\choose 2} \rightarrow [l]\right\}$$
and the edges are of the form $\{\phi, \{i,j,k\}_{abc}\}$ if and only if $\phi \in \Phi$,
 $\{i,j,k\}_{abc}\in \cJ^C$ where $\phi(\{i,j\})=a,\  \phi(\{j,k\})=b,\  \phi(\{i,k\})=c$.
 Let $E$ denote the number of edges in this bipartite graph. Since each $\{i,j,k\}_{abc} \in \cJ^C$ has degree precisely
$l^{{t \choose 2}-3}$, we have \begin{equation}\label{uppE}
E=|\cJ^C|
l^{{t \choose 2}-3}.\end{equation} Note that the degree of $\phi$ is
$|\cJ_{\phi}|$.
 Suppose for contradiction that the number of $\phi$ for which
 $|\cJ_{\phi}| \ge (1-\nu')\frac{|\cJ^C|}{l^3}$ is less than $(1-\nu') l^{{t \choose 2}}$.
  By (\ref{jphi}), we have $|\cJ_{\xi}|\le \frac{t^3}{27}$ for each $\xi\in \Phi$ and hence
\begin{equation}\label{upptwoE}E\le (1-\nu') l^{{t \choose 2}}\frac{t^3}{27} + \nu'l^{{t \choose
2}}(1-\nu')\frac{|\cJ^C|}{l^3}.\end{equation} Using \eqref{uppE}
 and dividing by $l^{{t \choose 2}-3}$  yields
$$|\cJ^C| \le (1-\nu')l^3\frac{t^3}{27} + \nu'(1-\nu')|\cJ^C|.$$
After simplifying we obtain
$$(1-\nu'(1-\nu'))|\cJ^C|\le (1-\nu')l^3\frac{t^3}{27}.$$
  The lower bound $|\cJ^C|>(1-\nu)\frac{l^3t^3}{27}$ then gives
$$(1-\nu'(1-\nu'))(1-\nu)< 1-\nu'.$$
Since $\nu'=\nu^{1/4}$, the left hand side expands to
$$1-\nu'+\nu^{1/2}-\nu+\nu^{5/4}-\nu^{3/2}>1-\nu'.$$
This contradiction completes the proof.\end{proof}

{\bf Claim 1.}
$$|\cJ^C|\le (1-\nu)\frac{2l^3 t^3}{27}.$$
Once we have proved Claim 1, the proof is complete by the following
argument which is very similar to that in \cite{NR} and in \cite{t5}. Define
$$S^C=\bigcup_{ \{i,j,k\}_{abc} \in \cJ^C} \cK_3(P_a^{ij} \cup P_b^{jk} \cup P_c^{ik}).$$
The Triangle Counting Lemma implies that $|\cK_3(P_a^{ij} \cup
P_b^{jk} \cup P_c^{ik})| <\frac{m^3}{l^3}(1+\theta)$.  Now  Claim 1
and $(\ref{theta})$  give
$$|S^C| \le \frac{m^3}{l^3}(1+\theta)|\cJ^C|\le m^3(1+\theta)(1-\nu)\frac{2t^3}{27}<m^3\frac{2t^3}{27}
\left(1-\frac{\nu}{2}\right)\le
\frac{2n^3}{27}\left(1-\frac{\nu}{2}\right).$$
Since $\HH' \subset S^C$
for every $\HH \in C$,
$$|\{\HH': \HH \in C\}|\le 2^{(1-\frac{\nu}{2})\frac{2n^3}{27}}.$$

Each $\HH \in C$ can be written as $\HH=\HH' \cup \cE_{\HH}$. In view of (\ref{x}) and
$|\cE_{\HH}|\le \sigma n^3$, the number of $\cE_{\HH}$ with $\HH \in
C$ is at most $\sum_{i\le \sigma n^3} {n^3 \choose i}\le
2^{H(\sigma)n^3}$.
 Consequently,
$$|C| \le 2^{(1-\frac{\nu}{2})\frac{2n^3}{27}+H(\sigma)n^3}
$$
so (\ref{C}) holds and we are done.

{\bf Proof of Claim 1.} Suppose to the contrary that $|\cJ^C|>
(1-\nu)\frac{2l^3 t^3}{27}$.   We  apply  Lemma~\ref{fact}  and
conclude that for most functions $\phi$ the corresponding triple
system $\cJ_{\phi}$ satisfies
$$|\cJ_{\phi}| \ge (1-\nu')\frac{|\cJ^C|}{l^3} > (1-\nu')(1-\nu)\frac{2t^3}{27}>(1-2\nu')\frac{2t^3}{27}.$$
By  Theorem \ref{fps}, we conclude that for all of these $\phi$, the
triple system $\cJ_{\phi}$ has an ordered partition where the number
of inconsistent edges is at most  $\nu'' t^3$.  Let $\cG$ be the set of consistent edges of $\cJ_{\phi}$ and let $\cB$ be the set of inconsistent edges of $J_{\phi}$. Write $\cM$ for the set of consistent triples that are not edges of $\cJ_{\phi}$.  Then $\cG \cup \cM$ is semi-bipartite, so
$$|\cG|+|\cM| \le \max_{1 \le a \le t} {a \choose 2}(t-a)\le \frac{2t^3}{27}.$$
We also have $|\cG|+|\cB| =|\cJ_{\phi}|\ge (1-2\nu')\frac{2t^3}{27}$ and $|\cB|\le \nu'' t^3$.  Consequently,
\begin{equation} \label{nu''} |\cM| \le \left(\frac{4\nu'}{27}+\nu''\right)t^3<2\nu''t^3. \end{equation}

Fix one such $\phi$ and let the optimal partition of $\cJ_{\phi}$ be
$P_{\phi}=(X,Y)$. Since
$|\cJ_{\phi}| \ge (1-2\nu') \frac{2t^3}{27}$ and
$|D_{P_{\phi}}|\le \nu''t^3$, we obtain
$$|X|= (1 \pm \sqrt{\nu''}) \frac{2t}{3} \quad  \hbox{\rm and} \quad
|Y|= (1 \pm 2\sqrt{\nu''}) \frac{t}{3}.$$ Indeed, otherwise a short
calculation using (\ref{nu'}) gives
$$|\cJ_{\phi}|\le {|X| \choose 2}|Y| +|D_{P_{\phi}}|\le {(1-\sqrt{\nu''})\frac{2t}{3} \choose 2}(1+2\sqrt{\nu''})\frac{t}{3}+\nu'' t^3<(1-2\nu')\frac{2t^3}{27}.$$
\qed

Let $P=(V_X,V_Y)$ be the corresponding vertex partition of $[n]$,
obtained from the proof of Claim 1. In other words,
$$V_X=\bigcup_{i \in X} V_i \quad \hbox{ and } \quad
V_Y=\bigcup_{i \in Y} V_i.$$
  We will show that $P$ is a partition
of $[n]$ where the number of inconsistent edges $|D_P|$ is fewer
than $\eta n^3$.  This contradicts the fact that $\HH \in Forb(n, T_5)-Forb(n,
T_5, \eta)$ and completes the proof of Theorem \ref{stablet}.

From \eqref{rhoupp} $|\cE_{\HH}|\le \sigma n^3 \le
\frac{\eta}{2}n^3$ so it suffices to prove that $|D_P -\cE_{\HH}|\le
\frac{\eta}{2}n^3$.

Call a $\xi: \binom{[t]}{2} \rightarrow [l]$ {\it good} if it satisfies the
conclusion of Lemma~\ref{fact} (i.e. $|\cJ_{\xi}|>(1-\nu')\frac{|\cJ^C|}{l^3}$), otherwise call it {\it bad}. For each $\xi$ and
edge $\{i,j,k\} \in \cJ_{\xi}$, we have $a,b,c$ defined by
$a=\xi(\{i,j\})$ etc. let $\HH_{\xi}$ be the union, over all
$\{i,j,k\} \in \cJ_{\xi}$, of the edges of $\HH$ that lie on top of
the triangles in $P_{a}^{ij} \cup P_{b}^{jk} \cup P_{c}^{ik}$.  Let
$D_{\xi}$ be the set of edges in $\HH_{\xi}$ that are inconsistent
with respect to $P=(V_X, V_Y)$. We will estimate $|D_P-\cE_{\HH}|$
by summing  $|D_{\xi}|$ over all $\xi$.  Please note that each $e\in
D_P-\cE_{\HH}$  lies in exactly  $l^{{t \choose 2}-3}$ different
$D_{\xi}$ due to the definition of $\cJ^C$. Summing over all $\xi$
gives
$$l^{{t \choose 2}-3} |D_P-\cE_{\HH}|
=\sum_{\xi:\binom{[t]}{2} \rightarrow [l]}|D_{\xi}|=\sum_{\xi\  good} |D_{\xi}|
+\sum_{\xi\  bad} |D_{\xi}|.$$ Note that for a given edge $\{i,j,k\}
\in \cJ_{\phi}$ the number of edges in $\HH_{\phi}$ corresponding to
this edge is the number of edges in $V_i \cup V_j \cup V_k$ on top
of triangles formed by the three bipartite graphs, each of which is
$\epsilon$-regular of density $1/l$.  By the Triangle Counting
Lemma, the total number of such triangles is at most
$$2|V_i||V_j||V_k|\left(\frac1l\right)^3<2\left(\frac{n}{t}\right)^3 \left(\frac1l\right)^3.$$
 By Lemma~\ref{fact} the number of bad $\xi$ is at most $\nu' l^{{t\choose 2}}$.  So we have
$$\sum_{\xi\  bad} |D_{\xi}|\le \nu' l^{{t\choose 2}}
{t \choose 3}2\left(\frac{n}{t}\right)^3
\left(\frac1l\right)^3<\nu'l^{{t\choose 2}-3}n^3.$$ It remains to
estimate $\sum_{\xi\  good} |D_{\xi}|$.

 Fix a good $\xi$ and let the optimal partition of $\cJ_{\xi}$ be $P_{\xi}=A \cup B$
 (recall that  $|D_{P_{\xi}}|\le \nu''t^3$, $A=(1 \pm \sqrt{\nu''})\frac{2t}{3}$, $B=(1 \pm 2\sqrt{\nu''})\frac{t}{3}$).

{\bf Claim 2.} The number of consistent edges of $\cJ_{\xi}$ with
$P_{\xi}$ that are inconsistent edges of $\cJ_{\phi}$ with
$P_{\phi}$ is at most $4(\nu'')^{1/3}t^3$.

Suppose that  Claim 2 is true. Then
$$\sum_{\xi\ good} |D_{\xi}| \le l^{{t\choose 2}}\left[4(\nu'')^{1/3}t^3(\frac{n}{t})^3\frac{2}{l^3}+
\nu''t^3(\frac{n}{t})^3\frac{2}{l^3}\right]= l^{{t\choose
2}-3}\left[10(\nu'')^{1/3}n^3\right].$$ Explanation: We consider the
contribution from the inconsistent edges of $P_{\phi}$ that are (i)
consistent edges of $P_{\xi}$ and (ii) inconsistent edges of
$P_{\xi}$. We do not need to consider the contribution from the
consistent edges of $P_{\phi}$ since by definition, these do not
give rise to edges of $D_P$.

Altogether, using (\ref{nu''def}) and (\ref{nu'}) we  obtain
$$|D_P-\cE_{\HH}| \le (10(\nu'')^{1/3}+\nu')n^3<\frac{\eta}{2} n^3$$
and the proof is complete.  We now  prove Claim 2.

{\bf Proof of Claim 2.} First we argue that for every $A' \subset A,
B' \subset B$ with $\min\{|A'|, |B'|\}\ge 3(\nu'')^{1/3}t$, the
number of edges in $\cJ_{\xi}$ with two points in $A'$ and one point
in $B'$ is at least $10\nu''t^3$.  Indeed, ${|A'|\choose
2}|B'|>12\nu''t^3$, and the number of triples with two points in $A'$
and one point in $B'$ that are not edges of $J_{\xi}$ is at most $2\nu''t^3$
by (\ref{nu''}). The remaining triples are edges in $\cJ_{\xi}$ with
two points in $A'$ and one point in $B'$  as desired.

Now suppose that $A'=A \cap Y$ and $B'=B \cap Y$ satisfy
$\min\{|A'|, |B'|\}\ge 3(\nu'')^{1/3}t$.  Then we have at least
$10\nu''t^3$ edges $e \in \cJ_{\xi}$ with $|e \cap A'|=2$ and $|e
\cap B'|=1$.  For each such edge $e=\{k,k',k''\}\subset Y$, and each
$\{i,j\} \in {X\choose 2}$, consider three distinct triples
$f=\{i,j,k\},f'=\{i,j,k'\}, f''=\{i,j,k''\}$ that are consistent
with $P_{\phi}$. If $f, f', f'' \in J_{\phi}$ then consider the
following ten bipartite graphs:

$$G^{ij}=P_{\phi(\{i,j\})}^{ij}, \quad G^{jk}=P_{\phi(\{j,k\})}^{jk}, \quad G^{ik}=P_{\phi(\{i,k\})}^{ik},$$
$$G^{ik'}=P_{\phi(\{i,k'\})}^{ik'},\quad
G^{jk'}=P_{\phi(\{j,k'\})}^{jk'},\quad
 G^{ik''}=P_{\phi(\{i,k''\})}^{ik''},\quad
G^{jk''}=P_{\phi(\{j,k''\})}^{jk''},$$
 $$G^{kk'}=P_{\xi(\{k, k'\})}^{kk'},\quad G^{k'k''}=P_{\xi(\{k', k''\})}^{k'k''},\quad G^{kk''}=P_{\xi(\{k, k''\})}^{kk''}.$$
Set $G=\bigcup G^{uv}$ where the union is over the ten bipartite
graphs defined above.
 Since $\{e,f,f', f''\} \subset \cJ_{\phi} \cup \cJ_{\xi}$,  the $3$-graph  $J=\{e, f,f', f''\}$ associated with the
 $5$-partite graph $G$ and $3$-graph $\HH(\{i,j,k,k',k''\})$ is a cluster $3$-graph.
 By our choice of constants in (\ref{delta}), we may apply the Embedding Lemma. As $J \cong T_5$,
   we obtain the contradiction  $T_5 \subset
\HH$. We conclude that $g\not\in \cJ_{\phi}$ for some $g\in \{f, f',
f''\}$.  Each $e$ gives rise to at least ${|X| \choose
2}>\frac{t^2}{5}$ such $g$ and each $g$ is counted by at most
$|Y|^2<\frac{t^2}{8}$ different $e$. Altogether we
obtain at least
$$\frac{10\nu'' t^3 \times \frac{t^2}{5}}{\frac{t^2}{8}} > 2\nu'' t^3$$
distinct triples $g$ that are consistent with  $P_{\phi}$ but are
not edges of $\cJ_{\phi}$.  This contradicts (\ref{nu''}) and we may
therefore suppose that either $|A \cap Y|<3(\nu'')^{1/3}t$ or $|B
\cap Y|<3(\nu'')^{1/3}t$.

Next suppose that $A'=A \cap X$ and $B'=B \cap X$ satisfy
$\min\{|A'|, |B'|\}\ge 3(\nu'')^{1/3}t$.  Then we have at least
$10\nu''t^3$ edges $e \in \cJ_{\xi}$ with $|e \cap A'|=2$ and $|e
\cap B'|=1$.  For each such edge $e=\{k,k',k''\}\subset X$, and each
$(i,j) \in (X-e)\times Y$, consider three distinct triples
$f=\{i,j,k\},f'=\{i,j,k'\}, f''=\{i,j,k''\}$ that are consistent
with $P_{\phi}$. If $f, f', f'' \in J_{\phi}$ then consider the ten
bipartite graphs defined above. Set $G=\bigcup G^{uv}$ where the
union is over these ten bipartite graphs.
 Since $\{e,f,f', f''\} \subset \cJ_{\phi} \cup \cJ_{\xi}$, the 3-graph  $J=\{e, f,f', f''\}$ associated with the
 $5$-partite graph $G$ and $3$-graph $\HH(\{i,j,k,k',k''\})$ is a cluster $3$-graph.
Again, by the Embedding Lemma we obtain the contradiction  $T_5 \subset
\HH$. We conclude that $g\not\in \cJ_{\phi}$ for some $g\in \{f, f',
f''\}$.  Each $e$ gives rise to at least $(|X|-3)|Y|>\frac{t^2}{5}$
such $g$ and each $g$ is counted by at most
$|X|^2 < \frac{t^2}{2}$ different $e$. Altogether we obtain at least
$$\frac{10\nu'' t^3 \times \frac{t^2}{5}}{\frac{t^2}{2}} > 2\nu'' t^3$$
distinct triples $g$ that are consistent with  $P_{\phi}$ but are
not edges of $\cJ_{\phi}$.  This contradicts (\ref{nu''}).

We may therefore suppose that

(i) $|A \cap Y|<3(\nu'')^{1/3}t$ or $|B \cap Y|<3(\nu'')^{1/3}t$ and

(ii) $|A \cap X|<3(\nu'')^{1/3}t$ or $|B \cap X|<3(\nu'')^{1/3}t$.

Let us now show that (i) and (ii) imply that
\begin{equation} \label{abxy}
|A \cap Y| + | B \cap X| <6(\nu'')^{1/3} t.\end{equation}
If $|A \cap Y| \ge 3(\nu'')^{1/3}t$, then by (i) we have $|B \cap Y| < 3(\nu'')^{1/3}t$. Consequently,
$$|A \cap X|=|A-Y|\ge |A|-|Y|\ge (1-\sqrt{\nu''})\frac{2t}{3}-(1+2\sqrt{\nu''})\frac{t}{3}>3(\nu'')^{1/3}t$$
and also
$$|B \cap X|=|B-(B \cap Y)|\ge (1-2\sqrt{\nu''})\frac{t}{3}-3(\nu'')^{1/3}t>3(\nu'')^{1/3}t.$$
This contradicts (ii) so we may assume that $|A \cap Y| < 3(\nu'')^{1/3}t$.

If $|B \cap X| \ge 3(\nu'')^{1/3}t$, then by (ii), we have $|A \cap X| <  3(\nu'')^{1/3}t$. This yields the contradiction
$$|X|=|A\cap X|+| B \cap X|< 3(\nu'')^{1/3}t +|B|<3(\nu'')^{1/3}t+(1+2\sqrt{\nu''})\frac{t}{3}
<(1-\sqrt{\nu''})\frac{2t}{3}.$$
We may therefore also assume that $|B \cap X| < 3(\nu'')^{1/3}t$ and now (\ref{abxy}) follows.

A consistent edge of $P_{\xi}$ that is inconsistent with $P_{\phi}$
must have a point in $(A \cap Y)\cup  (B \cap X)$, hence the number
of such edges is at most $6(\nu'')^{1/3} t{t \choose
2}<4(\nu'')^{1/3}t^3$ as required.  \qed

\section{Proof of Theorem \ref{maint}}\label{sectproofmain}

\subsection{Preliminaries}\label{subsecpre}

 Recall that  the binary entropy function $H(x):=
x\log_2 1/x+ (1-x)\log_2 1/(1-x).$
We shall use Chernoff's inequality in the form below:

\begin{theorem}\label{chernoff}
Let $X_1,\ldots,X_m$ be independent $\{0,1\}$ random variables with
$P(X_i=1)=p$ for each $i$. Let $S=\sum_i X_i$. Then the following
inequality holds for
$a>0$:\\
$$P(S < \eee S - a) < \exp(-a^2/(2pm)).$$
\end{theorem}

We shall also need the following easy statements.

\begin{lemma}\label{matching}
Every graph $G$ with $n$ vertices contains a matching of size at
least $\frac{|G|}{2n}$.
\end{lemma}

\subsection{Lower-density}\label{sublower}

Our goal in this section is twofold: First to define a subset
$Forb(n,T_5,\eta,\mu)\subset Forb(n,T_5,\eta)$ which comprises
3-graphs with ordered partitions $(X,Y)$ that have a collection
of useful properties. Second, to prove that most 3-graphs in
$Forb(n,T_5,\eta)$ are in $Forb(n,T_5,\eta,\mu)$.

 Let $\HH\in
Forb(n,T_5,\eta)$ and let $(X,Y)$ be an ordered partition of
the vertices of $\HH$ which minimizes the number of inconsistent
edges. We call such a partition {\it optimal}. For a vertex $x$
let $L_{X,X}(x)$ be the set of edges containing $x$,
 and having the other two vertices in $X$, and let
 $ L_{X,Y}(x)$ and $L_{Y,Y}(x)$ be similarly defined. Sometimes, trusting that it will not cause confusion, we refer to
 $L_{X,X}(x)$ as the link graph of $x$ on $X$.  As before, we often associate a graph or hypergraph with its edge set.

\begin{definition}\label{lowdef} An ordered partition $(X,Y)$ is $\mu$-{\it lower-dense} if each  of
 the following is satisfied:\\
 (i) For every matching $G_1\subset {X\choose 2}$ and every graph
 $G_2\subset X\times Y$ with $|G_1|>\mu n, |G_2|> \mu n^2$ the
 following holds:
$$|\{(ab,uv):\ ab\in G_2, uv\in G_1, abu, abv\in \HH  \}|>\frac{|G_1||G_2|}{72}.$$
(ii) For every graph $G_1\subset {X\choose 2}$ and every matching
 $G_2\subset {Y\choose 2}$ with $|G_1|>\mu n^2, |G_2|> \mu n$ the
 following holds:
$$|\{(ab,uv):\ ab\in G_2, uv\in G_1, auv, buv\in \HH  \}|>\frac{|G_1||G_2|}{8}.$$
(iii) For every $A_X\subset X, A_Y\subset Y$ with $|A_X|,|A_Y|\ge \mu n$
the following holds:
$$|\{E\in \HH: \ |E\cap A_X|=2, |E\cap A_Y|=1 \}|> \frac{|A_X|^2 |A_Y|}{8}.  $$
(iv) Let $Y'\subset Y$ with $|Y'| \ge 2\mu n$, and suppose that for
every
$y\in Y'$ we have an $X_y\subset X$ with $ |X_y|> 200 \mu n$. Then
$$|\{E\in \HH: \ \exists y\in Y' \text{ s.t. } |E\cap X_y|=2, y\in E \}|> 10000\mu^3 n^3.  $$
(v) $||Y|-n/3|<\mu n.$\\

We say that an $\HH\in Forb(n,T_5,\eta)$ is $\mu$-{\it lower-dense}
if each of its optimal partitions  satisfies conditions (i)-(v). Let
$Forb(n,T_5,\eta,\mu)\subset Forb(n,T_5,\eta)$ be the collection
$\mu$-{lower-dense} hypergraphs.
\end{definition}

\begin{lemma}\label{t5lower}
Let  $1000H(\eta)< \mu^3$ and $\mu$ be sufficiently small. Then for
$n$ sufficiently large
$$|Forb(n,T_5,\eta)-Forb(n,T_5,\eta,\mu)|\ <\
2^{n^3(\frac{2}{27}-\frac{\mu^3}{500})}.$$
\end{lemma}

\begin{proof}
We count the number of hypergraphs $\HH\in
Forb(n,T_5,\eta)-Forb(n,T_5,\eta,\mu)$ violating conditions (i)-(v)
 separately: We shall use the following estimates in many of the
 cases.
The number of ways to choose an ordered partition of $\HH$ is at
most $2^n$. In what follows let us assume that we are given such a
partition $(X, Y)$.  The number of ways the at most $\eta n^3$ inconsistent
edges could be placed is at most $2^{H(\eta)n^3}$, the number of
ways a subset of vertices could be chosen is at most $2^n$, the
number of ways a matching (of graph edges) could be chosen is at most $2^{n\log n}$,
and the number of ways a graph could be chosen is at most $2^{n^2}.$  The number of ways the consistent
edges could be chosen is at most $2^{\frac{|X|^2}{2}|Y|}\le 2^{2n^3/27}$. For this last bound, we will give some improvements using the fact that $\HH$ is not  $\mu$-lower-dense.

 For a fixed partition  of the vertex set,
 we may view the consistent edges as a probability space, where we choose each of them, independently, with probability $1/2$.
We use  Chernoff's inequality to show that the probability that a
particular condition of the definition of  $\mu$-lower density is
violated is low, yielding an upper bound on the number of ways of
choosing
the consistent edges of $\HH$.\\

(i) Given  the choice of $G_1$ and $G_2$,
 there are  $|G_1||G_2| \ge \mu^2 n^3$
possible pairs of  edges to be included mentioned in the
condition. However not all the edges are distinct, for example
if $u_1v_1, u_2v_2$ are edges in $G_1$ and $u_1b,u_2b$ are edges in $G_2$ then
the triple $u_1u_2b$ is considered for two pairs of edges:
$(u_1b,u_2v_2)$ and $(u_2b, , u_1v_1)$. In order to avoid this overcounting (which manifests itself as a lack of independence in a probability calculation) we shall choose subgraphs $G_1'\subset G_1,
G_2'\subset G_2$, such that $G_1',G_2'$ are vertex disjoint, and
$|G_1'|\ge \mu n/3, |G_2'|\ge \mu n^2/3.$

 We prove the existence of such $G_1'$ and $G_2'$ by randomly picking each edge of the matching $G_1$ with probability $1/2$,
  where these choices are independent for distinct edges.  Let $H_1$ be the (random) set of edges that were picked.
  Let $H_2$ be the (random) set of edges of $G_2$ that are disjoint from all edges of $H_1$.
   Then  $|H_1|$ is a binomial random variable with parameters $|G_1|$ and 1/2 and $|H_2|$ dominates a binomial
    random variable with parameters $|G_2|$ and 1/2.  The reason for this is that for $e \in G_2$, the probability that
    $e \in H_2$ is 1/2 or 1, depending on whether $e$ is incident to an edge of $G_1$ or not.
  So by Chernoff's inequality,
$$P(|H_i| < |G_i|/3)=P(|H_i| < |G_i|/2 -|G_i|/6)<\exp(-|G_i|/36)<\frac{1}{2}.$$
Consequently, $$P(|H_1| \ge |G_1|/3 \hbox{ and } |H_2| \ge |G_2|/3)>0$$
 and there exist $G_1'$ and $G_2'$ as above.

For each $uv \in G_1'$ and $ab \in G_2'$ let $X_{ab, uv}$ be the random variable that is 1 if both $abu, abv \in \HH$ and 0 otherwise.  Then $P(X_{ab, uv}=1)=1/4$, and since $G_1'$ and $G_2'$ are vertex disjoint, these random variables are independent.
We apply Chernoff's
inequality to these $m=|G'_1||G'_2|$ random variables with
 $a=m/8$ and $p=1/4$.  For $S=\sum_{uv \in G_1', ab \in G_2'}X_{ab, uv}$ this gives
 $$P\left(S\le \frac{|G_1||G_2|}{72}\right)\le P(S\le m/8) \le \exp\left(-\frac{(m/8)^2}{(m/2)}\right)=\exp(-m/32)<\exp\left(-\frac{\mu^2 n^3}{9\cdot 32}\right).$$

  Using this upper bound  we obtain that the number of hypergraphs that violate condition (i) is upper
bounded by
 $$2^{n+H(\eta)n^3+ n\log n+n^2+2n^3/27}\exp(-\mu^2 n^3/(9\cdot 32))< 2^{2n^3/27-\mu^2n^3/300}.$$

(ii) Given  the choice of $G_1$ and $G_2$,
 there are  $|G_1||G_2| \ge \mu^2 n^3$
possible pairs of  edges to be included mentioned in the
condition. Unlike in case (i), here all the edges are distinct so we do not need to construct $G_i'$.

For each $uv \in G_1$ and $ab \in G_2$ let $X_{ab, uv}$ be the random variable that is 1 if both $uva, uvb \in \HH$ and 0 otherwise.  Then $P(X_{ab, uv}=1)=1/4$, and these random variables are independent.
We apply Chernoff's
inequality to these $m=|G_1||G_2|$ random variables with
 $a=m/8$ and $p=1/4$.  For $S=\sum_{uv \in G_1, ab \in G_2}X_{ab, uv}$ this gives
 $$P\left(S\le \frac{|G_1||G_2|}{8}\right)\le P(S\le m/8) \le \exp\left(-\frac{(m/8)^2}{(m/2)}\right)=\exp(-m/32)
 \le \exp\left(-\frac{\mu^2 n^3}{32}\right).$$

 Using this upper bound  we obtain that the number of hypergraphs that violate condition (ii) is upper
bounded by
 $$2^{2n^3/27-\mu^2n^3/32}.$$
 (iii) Given  the choice of $A_X$ and $A_Y$,  there are
$|A_X|(|A_X|-1)|A_Y|/2 \ge \mu^3 n^3/3=:m$ possible edges  of $\HH$ with
two vertices in $A_X$ and one in $A_Y$.
 Using  Chernoff's inequality (with
$p=1/2$) we obtain that the number of hypergraphs violating
condition (iii) is at most   $$2^{3n+H(\eta)n^3+ 2n^3/27}\exp(-\mu^3
n^3/24)<2^{2n^3/27-\mu^3 n^3/24}.$$
(iv)  Given the ordered $2$-partition,  there are at most $2^n$
choices for each of $X_y$ and of $Y'$. Also  $$\left|\left\{E\in
\binom{[n]}{3}: \ \exists y\in Y' \text{ s.t. } |E\cap X_y|=2, y\in
E \right\}\right| \ge 2\mu n {200 \mu n \choose 2}>
 35000\mu^3 n^3.$$ By  Chernoff's inequality we obtain
that the number of hypergraphs violating condition (iv) is at
most
$$2^{2n^2+ H(\eta)n^3+ 2n^3/27}\exp(-\mu^3 n^3)<2^{2n^3/27-\mu^3 n^3}.$$
Note that in the computation above we used  $1000H(\eta)< \mu^3$ and
$n$ is sufficiently large.

(v) In this case we show that if ratio of the parts of the ordered partition differ too
much from  2, then the number of ways to place the consistent edges
decreases exponentially.  This is simply because the number of consistent edges is small. More precisely,
 if $||Y|-n/3|\ge \mu n$ then the number of possible consistent edges is
at most $(2/27-\mu^2/2+\mu^3/2)n^3<(2/27-\mu^2/3)n^3$. This implies that the number of such hypergraphs is at most
$$2^n\cdot 2^{n^3(2/27+H(\eta)-\mu^2/3)}<2^{n^3(2/27-\mu^2/6)}.$$
Summing up the number of 3-graphs  in cases (i)--(v) gives
\begin{eqnarray*}& & |Forb(n,T_5,\eta)-Forb(n,T_5,\eta,\mu)| \\
&\le& 2^{2n^3/27} \left(2^{-\mu^2n^3/300}+  2^{-\mu^2n^3/32}
+2^{-\mu^3 n^3/24} +2^{-\mu^3 n^3}+2^{-\mu^2/6}\right)\\ &<&
2^{2n^3/27 -\mu^3n^3/500}.\end{eqnarray*}
This completes the proof of the lemma.
\end{proof}

\subsection{Getting rid of bad vertices}\label{subbadvert}

From now on we shall have the following hierarchy of constants:
$1\gg \alpha\gg \beta\gg \mu\gg \eta$. More precisely we will assume
\begin{equation}\label{parameters} 0.01>H(\alpha),\  \ \alpha^2
>100( H(\beta)+H(2\mu)+\mu ^2),\ \beta> 100H(2\mu),\ \mu^3\ge 1000
H(\eta).\end{equation}

In this section we prove additional properties of hypergraphs in
$Forb(n,T_5,\eta,\mu)$ which involve the link graph of vertices.

\begin{lemma}\label{badvertices}
Let $\HH\in Forb(n,T_5,\eta,\mu)$ with an  optimal ordered partition
$(X,Y)$.
Then the following hold.\\
(i) For $x\in X$ we have $|L_{X,X}(x)|\le 2\mu n^2.$\\
(ii) For $y\in Y$ we have $|L_{X,Y}(y)|\le 2\mu n^2.$\\
(iii) For $y\in Y$ we have $\min\{|L_{X,X}(y)|,|L_{Y,Y}(y)|\}< 2\mu
n^2.$
\end{lemma}

We remark here  that the lack of similar bounds for $x\in X$ on $|L_{X,Y}(x)|$ makes
the proof of the main result  complicated.

\begin{proof}
(i) Assume that for some $x\in X$ we have $|L_{X,X}(x)|> 2\mu n^2.$
By the optimality of the partition we have $|L_{X,Y}(x)|> 2\mu n^2$
as well. By Lemma~\ref{matching} $L_{X,X}(x)$ contains a matching
$G_1$ of size at least $\mu n$. With $G_2=L_{X,Y}(x)$, using
property (i) of the definition of $\mu$-lower-density, there exists
an $ab\in G_2$ and $uv\in G_1$ such that $abu,abv\in \HH$. Together
with $abx$ and $uvx$, we obtain   $T_5$ in $\HH$, a contradiction.

 (ii) Assume that for
some $y\in Y$ we have  $|L_{X,Y}(y)|> 2\mu n^2.$ By the optimality
of the partition we have $|L_{X,X}(y)|> 2\mu n^2$ as well. By
Lemma~\ref{matching} $L_{X,X}(y)$ contains a matching $G_1$ of size
at least $\mu n$. With $G_2=L_{X,Y}(y)$, using property (i) of the
definition of $\mu$-lower-density, there exists an $ab\in G_2$ and
$uv\in G_1$ such that $abu,abv\in \HH$. Together with $aby$ and
$uvy$ we obtain a $T_5$ in $\HH$, a contradiction.

 (iii) Assume that for some $y\in Y$ we have
$|L_{X,X}(y)|,|L_{Y,Y}(y)|> 2\mu n^2.$ By Lemma~\ref{matching}
$L_{Y,Y}(y)$ contains a matching $G_2$ of size at least $\mu n$.
With $G_1=L_{X,X}(y)$, using property (ii) of the definition of
$\mu$-lower-density, there exists an $ab\in G_2$ and $uv\in G_1$
such that $auv,buv\in \HH$. Together with $aby$ and $uvy$  we obtain
a $T_5$ in $\HH$, a contradiction.
\end{proof}

For a set $S\subset [n]$ of size  two and for $A\subset [n]$, we
define $L_A(S)$ to be the set of vertices $v\in A$ such that $\{v\}\cup
S\in \HH$.
 We call an edge $xyz\in \HH$ $\alpha$-{\it rich} with respect to
an optimal partition $(X,Y)$ of $\HH$ if $x\in X$, $y,z\in Y$ and
$\max\{|L_{X}(x,y)|, |L_X(x,z)|\}> \alpha n$. The vertex $z$ is the
{\it poor} vertex of a rich edge if $|L_{X}(x,y)|\ge  |L_X(x,z)|$;
in case of a tie we can decide arbitrarily.

\begin{lemma}\label{richedges}
Let  $(X,Y)$ be an optimal ordered partition  of an $\HH\in Forb(n,T_5,\eta,\mu)$.
For  $\alpha\ge 200\mu$  the following holds:\\
(i) The number of distinct poor vertices of the $\alpha$-rich edges of $\HH$ is at most $2\mu n$.\\
(ii) For any vertex $x\in X$ the number of $\alpha$-rich edges
containing $x$ is at most $2\mu n^2$.
%(iii) $L_{Y,Y}(x)$ can be covered by at most $2\mu n$ vertices of $Y$.
\end{lemma}

\begin{proof}
(i) Assume not, i.e., let $\{x_iy_iz_i\}$ be $\alpha$-rich
edges for $i\in [\lceil 2\mu n\rceil ]$, where $x_i\in X$ and
 $x_iy_iz_i$ has poor vertex $z_i$ and the $z_i$'s are
different vertices. Let $Y'=\{z_1,\ldots,z_{\lceil 2\mu
n\rceil}\}$ and $X_{z_i}=L_X(x_i,y_i)$. As $y_i$ is not the poor
vertex of the rich edge $x_iy_iz_i$, we have $|X_{z_i}|>200\mu n$.
By condition (iv) of the definition of $\mu$-lower-density  there is
an $i$ such that for some $a,b\in L_X(x_i,y_i), \ abz_i\in \HH$. But
then $x_iy_iz_i, x_iy_ia, x_iy_ib, abz_i$ form a $T_5$ in $\HH$,  a
contradiction.

(ii) The number of rich edges containing a vertex $z\in Y$ and $x$ is at most $n$, hence if (ii) was false,
then there would be at least $2\mu n$ poor vertices in $Y$, contradicting (i).
\end{proof}

\subsection{Estimates on $S(n)$}\label{estsn}

In this section we give some estimates on $S(n)$.

\begin{lemma}\label{mon}
(i)
$$\log_2 (S(n))\ge \frac{2}{27} n^3-\frac{1}{9} n^2-\frac{1}{9} n.$$
(ii) For $n$ large enough:
 $$S(n)\ge  S(n-1)\cdot 2^{(2n^2-5n+1)/9}\ge  S(n-2)\cdot
2^{(4n^2-14n+9)/9} \ge  S(n-3)\cdot 2^{(6n^2-27n+28)/9}.$$
\end{lemma}

\begin{proof}
(i) We generate many semi-bipartite
3-graphs as follows: Partition $[n]$ into classes of sizes
$t=\lceil 2n/3\rceil$ and $n-t=\lfloor n/3\rfloor$, and add any
collection of consistent edges. A short calculation shows that
$${t \choose 2}(n-t)\ge \frac{2}{27} n^3-\frac{1}{9} n^2-\frac{1}{9} n$$
and the result follows.

(ii) It is sufficient to prove the first inequality. Given a
semi-bipartite 3-graph on $[n-1]$ with partition $(X,Y)$, add $n$
to $Y$ if $|Y|<n/3$ otherwise to $X$, and decide about each
consistent edge containing $n$ to be added to the 3-graph or not.
If $|Y|<2n/3$ then careful calculation shows that for a given
partition there are at least $2^{(2n^2-5n+2)/9}$ ways to add
consistent edges containing $n$. However, if $|Y|\ge 2n/3$ then
 we do not generate too many 3-graphs, indeed in this case the number of possible consistent edges is at most
 $${|X| \choose 2}|Y| \le {n/3 \choose 2}\frac{2n}{3}\le \frac{n^3}{27}.$$ Consequently, the number of semi-bipartite
  $3$-graphs with vertex set $[n-1]$ and $|Y| \ge 2n/3$ is at most $2^{n+n^3/27}<S(n-1)\cdot (1-2^{-1/9})$ for $n$ large enough
  by part (i).
 Therefore
$$S(n)>(S(n-1)-2^{n+n^3/27})2^{(2n^2-5n+2)/9}>S(n-1)\cdot
2^{(2n^2-5n+1)/9}.$$
\end{proof}

\subsection{Getting rid of the inconsistent edges}\label{secinconedges}

In this section  we estimate the number of $3$-graphs $\HH$ from
$Forb(n,T_5,\eta,\mu)$ which violate one of the conditions below.
Note that if an $\HH$ does not violate any of the conditions below
then $\HH\in \S(n)$.

(1) In every optimal partition  $(X,Y)$ of $\HH$ and for every
$x\in X$ we have $|L_{Y,Y}(x)|< \beta n^2$.\\
(2)  In every optimal partition  $(X,Y)$ of $\HH$ every $y\in Y$
satisfies $|L_{Y,Y}(y)|<2\mu n^2.$   \\
 (3) No  optimal partition  $(X,Y)$ of $\HH$ contains an $\alpha$-rich
edge.\\
(4) No optimal partition $(X,Y)$ of $\HH$ has an inconsistent edge
 $xyz$ with $|\{x,y,z\} \cap X|\in\{0,3\}$.\\
(5) No optimal partition $(X,Y)$ of $\HH$ has an inconsistent edge
 $xyz$ with $|\{x,y,z\} \cap X|=1$.
\\

Our goal is to prove the following result, which will be completed in the next section.

\begin{theorem}\label{badtriple}
There is a $C_1$ such that the number of $\HH\in
 Forb(n,T_5,\eta,\mu)$ not satisfying any of the conditions (1)-(5) is at most  $C_1\cdot
 2^{-n/10}S(n)$.
\end{theorem}

Before proceeding we state and prove the following lemma.
For integers $a<b$, let $[a,b]=\{a, a+1, \ldots, b\}$.

\begin{lemma}\label{matchcount}
Fix a matching $M$ with $m$ edges, say $\{1,2\},\ldots, \{2m-1, 2m\}$. The
number of graphs on $[N]$,  where $M$ is a maximum matching is less
than $$2^{2m^2-2m} (N-2m + 2^{N-2m+1})^{m}.$$
\end{lemma}

\begin{proof} We allow complete freedom to include edges on $[2m]$ yielding $2^{\binom{2m}{2}-m}=
2^{2m^2-2m}$ ways to choose these edges.
There is no edge inside $[2m+1,N]$  by the maximality of $M$. Consider an edge $\{2i-1, 2i\} \in M$.  If for  $j_1,j_2\in [2m+1,N]$ both
$\{j_1,2i-1\}$ and $\{j_2, 2i\}$ are edges then again by maximality of $M$, we have $j_1=j_2$.  So either there is a vertex in $[2m+1, N]$
 with
edges to both  $2i-1$ and $2i$, or one of $2i-1$ or $2i$ has no edge to any vertex
 in $[2m+1,N]$. For each $i$ we obtain $N-2m + 2^{N-2m+1}$  possibilities for the set of edges incident to $\{2i-1,2i\}$,
  thereby completing the proof.
\end{proof}

In the next five subsections, we will let $n$ be sufficiently large as needed.

\subsubsection{3-graphs  violating
(1)}  In this section we prove the following Lemma.

\begin{lemma} \label{1}
The number of  $\HH \in Forb(n,T_5,\eta,\mu)$ violating condition (1)
is at most
$$|Forb(n-1,T_5)|2^{\frac{2n^2}{9} - \frac{\beta n^2}{5}}.$$
\end{lemma}
\proof
First we  fix an optimal partition $(X,Y)$ of $\HH$, which
can be chosen in at most $2^n$ ways. Choose an $x\in X$, which can
be done in at most $n$ ways. Assume that $|L_{Y,Y}(x)|\ge \beta
n^2$. Let $$B:=\{z\in Y:\ \exists y\in Y \text{ s.t. } xyz \text{ is
}\alpha\text{-rich, where } z \text{ is the poor vertex of } xyz
\}.$$ By Lemma~\ref{richedges} (i) we have $|B|\le 2\mu n$. So $Y-B$
does not contain both $y$ and $z$ from  an $\alpha$-rich edge $xyz$.
Let $M \subset {Y-B \choose 2}$ be a maximum matching
in $L_{Y,Y}(x)$. Since $|Y|<n/2$ and $\beta>10\mu$, we have
  $$|M|\ge (|L_{Y,Y}(x)|-2\mu n^2)/2|Y|\ge
 \beta n/2.$$
  Denote the vertex set of the matching $M$  by $A$, and let  $m=|M|$.

   The number of choices for $A$  is at
   most $2^n$, and the number of choices for $M$
is at most $2^{n\log n}$.
 For every $y\in A$ we have
$|L_X(x,y)|<\alpha n^2.$ The number of choices for $\HH-x$ is at
most $|Forb(n-1,T_5)|$, and by Lemma \ref{badvertices} part (i) the number of choices for $L_{X,X}(x)$ is
at
most $\sum_{i\le 2\mu n^2} \binom{n^2}{i}\le 2^{H(2\mu)n^2}$. The number of choices for  the edges of
$L_{Y,Y}(x)$ intersecting $B$ is at most $2^{|B||Y|}< 2^{\mu n^2}$.
Using Lemma~\ref{matchcount}, given $M$, the number of ways the rest
of $L_{Y,Y}(x)$ can be chosen is at most
$$2^{2m^2-2m} (|Y|-2m + 2^{|Y|-2m+1})^{m}<2^{2m^2-2m} (2^{|Y|-2m+2})^{m}= 2^{|Y|m}.$$

Since $|Y|\le n/3+ \mu n$, the number of ways the consistent edges containing $x$ could be chosen is at most
 $2^{|X||Y|}<2^{2n^2/9+\mu n^2}$. Our goal is to improve this bound by using the fact  that $x$ violates condition (1).
  Specifically, we write $Y=A \cup (Y-A)$ and replace $2^{\frac{2n^2}{9}+\mu n^2}$ by
 $2^{\frac{2n^2}{9}+\mu n^2}\cdot 2^{-2m|X|}\cdot \ell$, where $\ell$ is the number of ways to
  add edges of the form $xab\in \HH$ with $a\in A, b\in X$.

 The number of ways to choose the (consistent) edges of the form
$xab\in \HH$ with $a\in A, b\in X$ is
$$\ell\le \left(\sum_{i\le \alpha n}\binom{|X|}{i}\right)^{2m}< 2^{2H(\alpha)mn}.$$ Here we use the fact that $a,x$ are
in a non-rich inconsistent edge, so for given $a$ this restricts the
number of choices for  $b$. To summarize, the number of 3-graphs
for given $m$ violating (1) is at most
\begin{equation}\label{part2} n2^n2^n 2^{n\log n}|Forb(n-1,T_5)|2^{H(2\mu)n^2}2^{\mu n^2}
2^{-2m|X|} 2^{|Y|m}2^{\frac{2n^2}{9}+\mu n^2}2^{2H(\alpha)mn}.
\end{equation}
The coefficient of $m$ in the exponent above is
$$-2|X|+|Y|+2H(\alpha)n<-1/2.$$
Therefore, viewing \eqref{part2} as a function of $m$, it is maximized when $m$
is minimized, i.e.~$m = \beta n/2$. Since $100H(2\mu)<\beta \ll 1$, we have
$$\frac{2\log n}{n^2}+\frac{2+\log n}{n} + H(2\mu)+ 2\mu-\beta/4 <-\beta/5.$$ Since there are at most $n$ choices for $m$, we conclude that the
number of 3-graphs violating (1) is bounded above by
\begin{equation}\label{part2s}
|Forb(n-1,T_5)|2^{\frac{2n^2}{9} - \frac{\beta n^2}{5}}
\end{equation}
as required. \qed

\subsubsection{3-graphs  violating
(2)}  In this section we prove the following Lemma.

\begin{lemma} \label{2}
The number of  $\HH \in Forb(n,T_5,\eta,\mu)$ violating condition (2)
is at most
$$|Forb(n-1,T_5)|\cdot 2^{\frac{n^2}{17}}.$$
\end{lemma}
\proof
 First
fix an optimal partition $(X,Y)$, which can be chosen at most $2^n$
ways. Given an optimal partition $(X,Y)$, assume that there is a
$y\in Y$ such that $|L_{Y,Y}(y)|\ge 2\mu n^2.$ Then by
Lemma~\ref{badvertices} (iii) we have $|L_{X,X}(y)|< 2\mu n^2$, and
by optimality of the partition $(X,Y)$ we have $|L_{X,Y}(y)|\le 2\mu
n^2$. So the number of 3-graphs having such a vertex $y$ is at
most
\begin{equation}\label{onebad}
n2^n |Forb(n-1,T_5)|\cdot 2^{{|Y|^2}/{2}+2H(2\mu)n^2}<
|Forb(n-1,T_5)|\cdot 2^{\frac{n^2}{17}},
\end{equation}
where we used  condition (iv) of Definition~\ref{lowdef}. \qed

\subsubsection{3-graphs  satisfying (1) and (2) but violating
(3)}  In this section we prove the following Lemma.

\begin{lemma} \label{3}
The number of  $\HH \in Forb(n,T_5,\eta,\mu)$ satisfying conditions (1) and (2) but violating condition (3)
is at most
$$|Forb(n-3,T_5)|2^{\frac{2n^2}{3} -\frac{\alpha^2n^2}{3}}.$$
\end{lemma}
\proof
 Assume that $(X,Y)$ is an optimal partition of
$\HH$ and $xyz$ is an $\alpha$-rich edge with $x\in X, y,z\in
Y$ and $|L_X(x,y)|\ge |L_X(x,z)|$. The edge $xyz$ could be chosen in
at most $n^3$ ways and $L_X(x,y)$ can be chosen in at most $2^n$
ways. Given these choices, we can choose $\HH-\{x,y,z\}$ in at most
$|Forb(n-3,T_5)|$ ways. By Lemma~\ref{badvertices}~(i) and the fact
that $\HH$ satisfies condition (1), the number of ways the
inconsistent edges  containing $x$ can be chosen is at most
$2^{H(2\mu)n^2+ H(\beta) n^2}$. By Lemma~\ref{badvertices}~(ii) and the fact that $\HH$ satisfies condition (2), the number of ways of having the
inconsistent edges intersecting $y$ or $z$ is at most $2^{4H(2\mu)n^2}$. The number of
ways the consistent edges containing $x$ or $y$ could be chosen is at
most $2^{|X|\cdot|Y|+{|X|^2}/{2}}$.
The number of
ways the consistent edges containing
 $z$ could be chosen is at most
$2^{\frac{|X|^2}{2}-{|L_X(x,y)|\choose 2}}$, as for $a,b\in
L_X(x,y)$, edge $abz$ together  with $xyz,xya,xyb$
 forms a copy of $T_5$. Since  $xyz$ is an $\alpha$-rich   $|L_X(x,y)|\ge \alpha n.$
 So the number of 3-graphs satisfying (1) and (2) but violating (3) is at most
\begin{equation} \label{123} 2^nn^3|Forb(n-3,T_5)|2^{H(2\mu)n^2+ H(\beta)n^2+4H(2\mu)n^2}
\cdot 2^{|X|  \cdot
|Y|+\frac{|X|^2}{2}+\frac{|X|^2}{2}-{|L_X(x,y)|\choose 2}}.\end{equation}
 Since $\alpha^2>100(H(\beta)+H(2\mu) +\mu^2)$,
 $$\frac{n+3\log n}{n^2}+  6 H(2\mu)+H(\beta)+\mu- \alpha^2/2+ \frac{\alpha}{n} <-\frac{\alpha^2}{3}.$$
  As $|X|\le 2n/3+\mu n$ and $|X||Y|+|X|^2=|X|n$, we conclude that (\ref{123}) is at most
\begin{equation}\label{partiii}
 |Forb(n-3,T_5)|2^{\frac{2n^2}{3} -\frac{\alpha^2n^2}{3}},
\end{equation} thereby completing the proof. \qed

\subsubsection{3-graphs  satisfying (1), (2) and (3) but violating
(4)}  In this section we prove the following Lemma.

\begin{lemma} \label{4}
The number of  $\HH \in Forb(n,T_5,\eta,\mu)$ satisfying conditions (1) and (2) and (3) but violating condition (4)
is at most
$$|Forb(n-3,T_5)|
2^{\frac{7n^2}{11}}.$$
\end{lemma}
\proof  First fix an optimal partition $(X,Y)$,
which can be chosen at most $2^n$ ways.
 Given an optimal partition $(X,Y)$,  an inconsistent edge $xyz$ could be chosen in
 at most $n^3$ ways. We can choose $\HH-\{x,y,z\}$ in
at most $|Forb(n-3,T_5)|$ ways. The number of  edges having at least
 two of $x,y,z$ is at most $3n$, giving at most $2^{3n}$ ways to place
them.

Now consider the case that $x,y,z\in X$.
  There are two
types of inconsistent edges $e$ containing one of $\{x,y,z\}$, either $e \subset X$, or $e-\{x,y,z\} \subset Y$.
In the first case Lemma~\ref{badvertices} (i) implies that there are at most
$3\cdot 2\mu n^2$ such edges, and in the second case,  since $\HH$ satisfies condition (1)  there are at most $3\cdot \beta n^2$ such edges.
So the number
of ways the  inconsistent edges  intersecting $\{x,y,z\}$ can be
chosen is at most
$$2^{(3H(\beta)+3H(2\mu)) n^2}.$$
The number of ways that the consistent edges containing any of
$x,y,z$  can  be chosen is restricted as follows:
   For any $a\in X, b\in Y$ out
of the $8$ possibilities including edges $abx,aby,abz$ only $7$ can
occur (all of them cannot be chosen at the same time), so the number
of possible connections is at most $7^{|X||Y|}$.

Consider now  the other case when $x,y,z\in Y$.  There are two types
of inconsistent edges $e$: Either $e\subset Y$ or $e\cap X\ne 0$. In
the first case, since $\HH$ satisfies  condition (2), that there are at most $3\cdot
2\mu n^2$ such $e$, and in the second case
  Lemma~\ref{badvertices} (ii) implies that there are at most $3\cdot 2\mu n^2$
  such $e$.
 So the number of ways to choose those edges is at most $2^{6H(2\mu)
n^2}<2^{(3H(\beta)+3H(2\mu)) n^2}.$ Now let us bound the number of
ways the consistent edges intersecting $\{x,y,z\}$ can be chosen.
Since
 for any pair  $a,b\in X$, we cannot have $\{abx,aby,abz\}\subset \HH$,
 the number of ways to place these type of edges is at most $7^{|X|^2/2}$.

Altogether the number of 3-graphs satisfying (1), (2) and (3) but violating (4) is bounded by
\begin{equation}\label{partiv}2^{n+1}n^3 2^{3n} 2^{{(3H(\beta)+3H(2\mu)) n^2}}|Forb(n-3,T_5)|
\left(7^{|X||Y|}+ 7^{\frac{|X|^2}{2}}\right).
\end{equation}
Since $\log_2 7< 2.81$, $\max\{|X||Y| ,|X|^2/2\}\le
(2/9+\mu)n^2 -1$, and
$$
\frac{n+1+3\log n+ 3n}{n^2} + 3H(\beta) +3H(2\mu)
+\frac{1}{n^2}<\frac{1}{100},
$$
(\ref{partiv}) is upper bounded by
$$|Forb(n-3,T_5)|
2^{\frac{7n^2}{11}}$$
as required. \qed

\subsubsection{3-graphs  satisfying (1), (2), (3) and (4) but violating
(5)}
Let us denote the 3-graphs $\HH$ described in the title of this section by $Forb^{(1)}(n,T_5,\eta,\mu)$.
Our goal in  this section is to prove the following Lemma.

\begin{lemma} \label{5}
The number of  $\HH \in Forb^{(1)}(n,T_5,\eta,\mu)$
is at most
$$(2^{-\alpha n^3}+2^{-n/10})S(n).$$
\end{lemma}

Lemma \ref{5} will be proved in several steps. First we need some
more definitions. Let $\HH\in Forb^{(1)}(n,T_5,\eta,\mu)$ and
$(X,Y)$ be an optimal partition of $\HH$. The {\it shadow-graph} of
the inconsistent edges with respect $(X,Y)$ is
$$G:=G_{\HH}(X,Y):=\bigcup_{y\in Y} L_{X,Y}(y).$$

Let $Forb(n,T_5,\eta,\mu,\alpha)\subset Forb^{(1)}(n,T_5,\eta,\mu)$
 be the collection of 3-graphs $\HH$  whose every optimal
partition $(X,Y)$ satisfies  $|G_\HH(X,Y)|<  100\alpha n^2$.

\begin{lemma}\label{clean2}
For $n$ sufficiently large
\begin{equation}\label{clean3} |Forb^{(1)}(n,T_5,\eta,\mu)- Forb(n,T_5,\eta,\mu,\alpha)| <
 2^{-\alpha n^3}S(n).\end{equation}
\end{lemma}

\proof
Let us count the number of $\HH\in Forb^{(1)}(n,T_5,\eta,\mu)-
Forb(n,T_5,\eta,\mu,\alpha)$. We can fix an optimal partition in at
most $2^n $ ways, and a shadow graph $G$ in at most $2^{n^2}$ ways.
As $\HH$ satisfies condition (3), there is no $\alpha$-rich edge of $\HH$.  Hence
for an edge $xy\in G$, there are at most $2^{H(\alpha)n}$ ways to choose
$L_X(x,y)$. Given $G$, the number of inconsistent
edges is at most $|G||Y|/2$ (each is counted twice). The number of consistent triples that are not edges is at least $|G|(|X|-\alpha n)/2$ for the following reason:  for each edge $xy \in G$, there is a vertex $z \in Y$ with $xyz \in \HH$. Since there is no $\alpha$-rich edge,
$|L_X(x,y)| \le \alpha n$, and so the number of consistent triples containing $x$ and $y$ that are not edges is at least $|X|-\alpha n$.
The factor two arises as these triples are counted at most twice.
Since ${|X| \choose 2}|Y| \le 2n^2/9$, we conclude that the number of consistent edges is at most
$$\frac{2n^3}{27}- \frac{|G|}{2}(|X|-\alpha n) \le \frac{2n^3}{27}
-\frac{|G||X|}{2}+\alpha n^3.$$
Each of these could  either be included in $\HH$ or
not. Altogether we obtain
\begin{eqnarray*}|Forb^{(1)}(n,T_5,\eta,\mu) - Forb(n,T_5,\eta,\mu,\alpha)|   &<& 2^n
2^{n^2} 2^{H(\alpha)n |G|} 2^{\frac{|G||Y|}{2}}2^{\frac{2n^3}{27}-
\frac{|G||X|}{2} +\alpha n^3}\\
&=&2^{\frac{2n^3}{27}-\frac{|G|(|X|-|Y|-2H(\alpha)n)}{2}+\alpha
n^3+ n^2+n}\\
&<&  2^{\frac{2n^3}{27}-2\alpha n^3}
\end{eqnarray*}
where the last inequality follows from
$|G|\ge 100\alpha  n^2$, $|X|-|Y|>n/4$ and
$H(\alpha)<0.01$.
The lower bound on $S(n)$ from Lemma~\ref{mon},
and  $n$  sufficiently large gives $S(n) > 2^{\frac{2n^3}{27}-\alpha n^3}$. Consequently,
$$|Forb^{(1)}(n,T_5,\eta,\mu) - Forb(n,T_5,\eta,\mu,\alpha)| \le
2^{-\alpha n^3}S(n)$$and the proof is complete. \qed

Now we shall show that the number of non-semi-bipartite 3-graphs
in $Forb(n,T_5,\eta,\mu,\alpha)$ is much smaller than the number of
semi-bipartite 3-graphs. First we partition
$Forb(n,T_5,\eta,\mu,\alpha)$ into $O(n^2)$ classes, and for each
class we construct a bipartite graph $B_i$.  One part of $B_i$ will be the
elements of a class $\C$, and the other part of $B_i$ will be the set of semi-bipartite
3-graphs $\SS(n)$.  $B_i$ will have the property that the degree of the
vertices in $\C$ will be exponentially larger than the degrees in
$\SS(n)$. This approach will allow us to prove the following Lemma.  Clearly Lemma \ref{clean2} and Lemma \ref{bipcom} immediately imply Lemma \ref{5}.

\begin{lemma}\label{bipcom}
For $n$ sufficiently large  $$|Forb(n,T_5,\eta,\mu,\alpha)-\SS(n)|<
2^{-n/10}S(n).$$
\end{lemma}

\begin{proof}
%For each $\HH \in Forb^{(2)}(n,T_5,\eta,\mu)-S(n)$ fix an optimal
%partition.
 For $i\le  100\alpha n^2$ let $\C_i\subset
Forb(n,T_5,\eta,\mu,\alpha)-\SS(n)$  be the collection of
3-graphs which have an optimal partition in which the shadow
graph of inconsistent edges has exactly $i$ edges. We construct
a bipartite graph $B_i$ with parts $\C_i$ and $\SS(n)$. An $\HH\in
\C_i$ will be joined in $B_i$ to the following set of
semi-bipartite 3-graphs, denoted by $\Phi(\HH)$:\\
- Remove all edges which contain an edge of $G$ (the shadow
graph of $\HH$) (so all the
inconsistent edges will be removed.)\\
- For every $xy\in G$ add some collection of edges $axy$ to $\HH$
where $a\in X$.

First we give a lower bound on the degree (in $B_i$) of a vertex  $\HH \in \C_i$. Here we have to give a lower bound on the number of
edges of the form $axy$ where $xy\in G$ (and say $y\in Y$).
Each edge can be counted at most twice,  so the number of
edges that we must decide to add to $\HH$  is at least
$(|X|-1)i/2$, therefore  $\text{deg}_{B_i}(\HH) \ge 2^{(|X|-1)i/2}$.

Before proceeding further we need the following.

{\bf Claim.}
Let $\HH\in \SS(n)$ such that $\Phi^{-1}(\HH)\ne\emptyset$. Then
 the number of partitions of  $[n]$ which are optimal partitions of $[n]$ is at most
$$2^{H(10\mu)n}.$$
{\bf Proof of Claim.}
If $\F\in \Phi^{-1}(\HH)$ then $\F\in Forb(n,T_5,\eta,\mu,\alpha)$
so it has a partition with at most $\eta n^3$ inconsistent edges.
Let $\F_j\in \Phi^{-1}(\HH)$ have an optimal partition $(X_j,Y_j)$
for $j=1,2$. We claim that $|X_1\Delta X_2|< 10\mu n$. Indeed,
otherwise w.l.o.g. $|X_1-X_2|\ge 5 \mu n$.
 Then by Definition~\ref{lowdef}~(v)
we have $||Y_1|-n/3|,\  ||Y_2|-n/3| < \mu n$ so $|X_2\cap Y_1|\ge
3\mu n$ and $|X_1\cap X_2|> n/4$. This makes it possible to find
many inconsistent edges inside $X_2$, as using
Definition~\ref{lowdef}~(iii)
$$|\{abc\in \HH: a,b\in X_1\cap X_2, c\in X_2\cap Y_1 \}|\ge
\frac{3}{16}\mu n^3 >\eta n^3. $$
 This contradiction shows that
the optimal partitions do not differ too much from each other. To
complete the proof of the Claim, we may count the number of optimal $(X_2, Y_2)$
by first picking the vertices of $|X_1\Delta X_2|$ and observing
that this determines $(X_2,Y_2)$.
\qed

Now we  fix an $\HH\in \SS(n)$, and give an upper bound on its
degree in the auxiliary graph.  Please recall that in forming $\HH$ we did not change any of the consistent edges that did not contain any edge of $G$. \\
 - The number of ways $G$ could be chosen is at most
 $\binom{n^2}{i}$.\\
 - Given $(X,Y)$ and $G$, the number of ways the inconsistent
 edges could be added is at most $2^{i|Y|/2}$.\\
 - Given $G$, and $xy\in G$, as $xy$ arises from
 an inconsistent edge that is not $\alpha$-rich,
  the number of consistent edges on $xy$ in the source 3-graph is at most
  $\alpha n$.  This gives at most  ${n \choose \alpha n}^i$ possibilities to choose the consistent edges that contain an edge of $G$.

By the Claim, the number of optimal partitions $(X, Y)$ is at most
$2^{H(10\mu)n}$.
So for each $\HH\in \S(n)$ we have
$$ \text{deg}_{B_i}(\HH) \le
 2^{H(10\mu)n} \binom{n^2}{i} 2^{i|Y|/2}\binom{n}{\alpha n}^i\le \left(2^{10H(\mu)+6\log n +|Y|/2 + H(\alpha) n }\right)^i.$$

Trivially, $|\C_i|/|\SS(n)|$ is at most the ratios of the bounds of the
degrees, i.e.,
$$\frac{|\C_i|}{S(n)}\le \left(2^{10H(\mu)+6\log n +|Y|/2 + H(\alpha) n  -|X|/2 +
1/2}\right)^i.$$
Since $||Y|-n/3| \le \mu n$, and $\mu$ is sufficiently small, $|X|-|Y| \ge n/3 -2\mu n \ge n/4$. Consequently, the expression above is upper bounded by $2^{-in/9}$.
We conclude that
$$|Forb(n,T_5,\eta,\mu,\alpha)-\SS(n)| \le \sum_{i=1}^{100\alpha n^2} |\C_i| \le n^2 S(n) 2^{-n/9}< S(n) 2^{-n/10}$$
and the proof is complete.
\end{proof}

\subsection{Completing the proofs of Theorems \ref{maint}, \ref{clean} and \ref{badtriple}}

In this section we will simultaneously prove Theorems \ref{maint}, \ref{clean} and \ref{badtriple} by induction on $n$. Write Theorem $P(n)$ for the statement that Theorem $P$ holds for $n$.
Also, let Theorem \ref{clean}$(\eta, n)$ denote the statement that Theorem \ref{clean} holds for $n$ with input parameter $\eta$.

 Let us first choose $\eta>0$ sufficiently small so that the hierarchy of the parameters in (\ref{parameters}) holds and $\eta$ is a valid input parameter for
 Theorem \ref{clean}.
 The structure of the induction arguments in the three proofs is as follows:
$$\hbox{Theorem } \ref{maint}(n-1) \longrightarrow \hbox{Theorem } \ref{badtriple}(n)
\longrightarrow \hbox{Theorem } \ref{clean}(\eta, n) \longrightarrow \hbox{Theorem } \ref{maint}(n).$$
The above will prove that Theorems \ref{maint} and \ref{badtriple} hold, and that Theorem \ref{clean} holds with input $\eta$. Since this is proved for each $\eta>0$ that is sufficiently small, it also proves Theorem \ref{clean}.

 With input parameter $\eta$, Theorem \ref{stablet} outputs $\nu$ and $n_0$.
Let  $n_1>n_0$ be sufficiently large such that for every $n>n_1$
Lemmas \ref{t5lower}, \ref{mon}, \ref{1}, \ref{2}, \ref{3}, \ref{4}
and \ref{5} hold. We also require $1/n_1$ to be much smaller than
all the fixed small constants in (\ref{parameters}). Let $c>100$ be
chosen so that Theorem \ref{badtriple} holds with $C_1=c$ for all $n
\le n_1$,  Theorem \ref{clean} with input $\eta$ holds with $C'=c$
for all $n\le n_1$ and Theorem~\ref{maint} holds with $C=c$ for all
$n \le n_1$. Now we fix
 $$C= 2C'= 4C_1=4c>400.$$

{\bf Proof of Theorem~\ref{badtriple}.}
We wish to prove Theorem \ref{badtriple}$(n)$, so as indicated above, we may assume Theorem \ref{maint}$(n')$ for $n'<n$.
We recall that if $\HH \in Forb(n, T_5, \eta, \mu) -  \S(n)$, then $\HH$ violates one of the conditions (1)--(5). Consequently, an upper bound for $|Forb(n, T_5, \eta, \mu) -  \S(n)|$ is obtained by summing the bounds in Lemmas \ref{1}--\ref{5}, which is
$$|Forb(n-1,T_5)|2^{2n^2/9 - \beta n^2/5}+ |Forb(n-1,T_5)|\cdot
2^{n^2/17} +
|Forb(n-3,T_5)|2^{6n^2/9-\alpha^2n^2/3}$$
$$ + \ |Forb(n-3,T_5)| 2^{7n^2/11} + (2^{-\alpha n^3} +
2^{-n/10})S(n).$$
We may assume that Theorem \ref{maint}$(n')$ holds for all $n'<n$ with parameter $C$.  Hence  we can upper bound this expression by
$$S(n-1)(C 2^{-(n-1)/10}+1)(2^{2n^2/9 - \beta
n^2/5}+2^{n^2/17})$$ $$+\  S(n-3)(C
2^{-(n-3)/10}+1)(2^{6n^2/9-\alpha^2n^2/3}+2^{7n^2/11}) +
S(n)(2^{-\alpha
n^3} + 2^{-n/10}).$$
Let us upper bound the terms above separately.
Since $n>n_1$, Lemma \ref{mon} (ii), yields $S(n-1) \le S(n) 2^{-(2n^2-5n+1)/9}$.
 As $\beta$ is sufficiently small (by (\ref{parameters})), we also have $ 2^{2n^2/9 - \beta n^2/5}>2^{n^2/17}$. Therefore
$$S(n-1)(C 2^{-(n-1)/10}+1)(2^{2n^2/9 - \beta
n^2/5}+2^{n^2/17})<S(n)(C 2^{-(n-1)/10}+1)2^{-\beta n^2/6}.$$
Similarly, using  $S(n-3) \le S(n) 2^{-(6n^2-27n+28)/9}$ and
$2^{6n^2/9-\alpha^2n^2/3}> 2^{7n^2/11}$ we obtain
$$S(n-3)(C
2^{-(n-3)/10}+1)(2^{6n^2/9-\alpha^2n^2/3}+2^{7n^2/11})<
S(n)(C 2^{-(n-3)/10}+1)2^{-\alpha^2n^2/4}.$$
 Summing up these bounds, we conclude that $|Forb(n, T_5, \eta, \mu) -  \S(n)|$ is upper bounded by
$$S(n)[(C 2^{-(n-1)/10}+1)2^{-\beta n^2/6} +
(C 2^{-(n-3)/10}+1)2^{-\alpha^2n^2/4}
 + \ 2^{-\alpha n^3}+ 2^{-n/10}].$$
After expanding the expression above, we see that each of the six
summands  is upper bounded by $\frac{C_1}{6}S(n)2^{-n/10}$ and we
finally obtain
 $$|Forb(n, T_5, \eta, \mu) -  \S(n)| \le S(n)C_1
2^{-n/10}.$$
This completes the proof. \qed
 \medskip

 {\bf Proof of Theorem~\ref{clean}.}
We wish to prove Theorem \ref{clean}$(\eta, n)$, so as indicated
above, we may assume Theorem \ref{badtriple}$(n)$. We also use
Lemma~\ref{t5lower}, Lemma~\ref{mon} (i) and $C'= 2C_1$:
 \begin{eqnarray*}
 |Forb(n,T_5,\eta)- \SS(n)| &\le &
 |Forb(n,T_5,\eta)- Forb(n,T_5,\eta,\mu)|
  \ +
|Forb(n,T_5,\eta,\mu)-\SS(n)|\\
& \le & 2^{n^3(2/27-\mu^3/500)} + C_1
2^{-n/10}S(n)\\ & \le & C_1
2^{-n/10}S(n)+C_1
2^{-n/10}S(n)\\
&=&C' 2^{-n/10}S(n).
 \end{eqnarray*}\hfill \qed

 {\bf Proof of Theorem~\ref{maint}.}
We wish to prove Theorem \ref{maint}$(n)$, so as indicated above, we
may assume Theorem \ref{clean}$(\eta, n)$. We also use
Theorem~\ref{stablet}, Lemma~\ref{mon} (i) and $C= 2C'$:
 \begin{eqnarray*}
|Forb(n,T_5)-S(n)|&\le & |Forb(n,T_5)-Forb(n,T_5,\eta)| +
|Forb(n,T_5,\eta)-\SS(n)|\\
& \le & \ 2^{(1-\nu){2n^3}/{27}}  + C'
2^{-n/10}S(n)\\
&\le & C'2^{-n/10}S(n)+C'2^{-n/10}S(n) \\
& = & C 2^{-n/10}S(n).\end{eqnarray*}
  \qed


\begin{thebibliography}{99}






\bibitem{BBS1} J.~Balogh, B.~Bollob\'as and M.~Simonovits,
On the number of graphs without forbidden subgraph, {\em J. Combin. Theory Ser. B}, \textbf{91} (2004), 1--24.

\bibitem{BBS2} J.~Balogh, B.~Bollob\'as and M.~Simonovits, The typical structure of graphs without given excluded
subgraphs, \emph{Random Structures and Algorithms},
     {\bf 34} {(2009)}, {305--318}.

\bibitem{BBS3} J.~Balogh, B.~Bollob\'as and M.~Simonovits, The fine structure of octahedron-free graphs, submitted.

\bibitem{t5} J. Balogh and D. Mubayi, Almost all cancellative triple systems
are tripartite, submitted.

\bibitem{BSAM} J. Balogh and W. Samotij, The number of $K_{s,t}$-free graphs, submitted.

\bibitem{BFMP} T. Bohman, A. Frieze, D. Mubayi and O. Pikhurko, Hypergraphs with independent neighborhoods, to appear in
Combinatorica.

\bibitem{bollobas:74}
B.~{Bollob\'as},
\newblock Three-graphs without two triples whose symmetric difference is
  contained in a third,
\newblock {\em Discrete Math.}, {\bf 8} (1974) 21--24.

\bibitem{BT1}  B.~Bollob\'as and A.~Thomason,  Projections of bodies and hereditary  properties of hypergraphs,
{\em Bull. London Math. Soc.}, {\bf 27} (1995) 417--424.

\bibitem{EFR} P.~Erd\H{o}s, P.~Frankl and V.~R\"odl, The asymptotic number of graphs not
containing a fixed subgraph and a problem for hypergraphs having no exponent, {\em Graphs and Combin.},
{\bf 2} (1986), 113--121.

\bibitem{EKR} P.~Erd\H{o}s, D.J.~Kleitman and B.L.~Rothschild, Asymptotic enumeration of $K_{n}$-free graphs,
in {\em  Colloquio Internazionale sulle Teorie Combinatorie} (Rome, 1973), Vol. II, pp. 19--27.
{\em Atti dei Convegni Lincei}, {\bf 17}, Accad. Naz. Lincei, Rome, 1976.



\bibitem{frankl+furedi:83}
P.~Frankl and Z.~F{\"u}redi,
\newblock A new generalization of the {E}rd{\H o}s-{K}o-{R}ado theorem,
\newblock {\em Combinatorica}, {\bf 3} (1983) 341--349.


\bibitem{FMP} Z. F\"uredi, D. Mubayi, and O. Pikhurko, Quadruple Systems with Independent Neighborhoods,
 {\em J. Combin. Theory Ser. A}, {\bf 115} (2008) 1552--1560.

\bibitem{FR} P. Frankl and V. R\"odl, Extremal problems on set systems,  {\em Random Structures and Algorithms},  {\bf 20}  (2002),
no. 2, 131--164.

\bibitem{FPS} Z. F\"uredi, O. Pikhurko and M. Simonovits, On Triple Systems with Independent Neighborhoods,
 {\em Comb, Prob and Comput}, {\bf 14} (2005) 795--813.


\bibitem{KM}
P.~Keevash and D.~Mubayi,
 Stability results for cancellative hypergraphs,
 {\em J. Combin. Theory Ser. B}, {\bf 92} (2004) 163--175.

\bibitem{KNRS}
Y. Kohayakawa, B. Nagle, V. R\"odl and M. Schacht, Weak hypergraph regularity and linear hypergraphs, {\em J.
Combin. Theory Ser. B}, in press (2009).

\bibitem{KPR} Ph.G.~Kolaitis, H.J.~Pr\"omel and B.L.~Rothschild, $K_{l+1}$-free graphs:
asymptotic structure and a $0$-$1$ law, {\em Trans. Amer. Math.
Soc.},  {\bf 303} (1987), 637--671.


\bibitem{MR} D. Mubayi and V. R\"odl, On the Tur\'an    number of triple systems, {\em J.  Combin. Theory, Ser.~A}, {\bf 100} (2002), no. 1, 136--152


\bibitem{NR} B. Nagle and V. R\"odl, The asymptotic number of triple systems not containing a fixed one,
 {\em Discrete Math.} {\bf 235} (2001), 271--290.

\bibitem{PSch} Y.~Person and M.~Schacht, Almost all hypergraphs without Fano planes are bipartite,
In: Claire Mathieu (editor): {\em Proceedings of the Twentieth Annual ACM-SIAM Symposium on Discrete Algorithms}
 (SODA 09), 217--226. ACM Press.

\bibitem{PS1} H.J.~Pr\"omel and A.~Steger, The asymptotic number of graphs
not containing a fixed color-critical subgraph,
  {\em Combinatorica},  {\bf 12} (1992) 463--473.
\end{thebibliography}
\end{document}